\documentclass[10pt]{amsart}

\usepackage{amsmath,amssymb,amsfonts}
\usepackage{mathrsfs,latexsym,amsthm,enumerate}
\usepackage{amsthm}
\input xypic
\theoremstyle{plain}
\newtheorem{theorem}{Theorem}[section]
\newtheorem{proposition}[theorem]{Proposition}
\newtheorem{corollary}[theorem]{Corollary}
\newtheorem{lemma}[theorem]{Lemma}

\newtheorem{remark}[theorem]{Remark}
\newcommand{\dom}{\mathbf{d}}
\newcommand{\ran}{\mathbf{r}}

\begin{document}

\title[Graph inverse semigroups]{Graph inverse semigroups:\\ their characterization and completion}

\author{D.~G.~Jones}
\address{Department of Mathematics,
Heriot-Watt University,
Riccarton,
Edinburgh~EH14~4AS
Scotland}
\email{dj26@ma.hw.ac.uk}
\thanks{The first author was supported by an EPSRC Doctoral Training Account reference EP/P503647/1,
and the second author's research was partially supported by EPSRC grants EP/F004184, EP/F014945, EP/F005881 and EP/I033203/1.
Some of the material in this paper appeared in the first author's PhD thesis \cite{DJ}.} 

\author{M.~V.~Lawson}
\address{Department of Mathematics
and the
Maxwell Institute for Mathematical Sciences, 
Heriot-Watt University,
Riccarton,
Edinburgh~EH14~4AS
Scotland}
\email{markl@ma.hw.ac.uk}


\begin{abstract} Graph inverse semigroups generalize the polycyclic inverse monoids and play an important role in the theory of $C^{\ast}$-algebras.
In this paper, we characterize such  semigroups and show how they may be completed, under suitable conditions, to form what we call the Cuntz-Krieger semigroup of the graph.
This semigroup is proved to be the ample semigroup of a topological groupoid associated with the graph, 
and the semigroup analogue of the Leavitt path algebra of the graph. 
\end{abstract}
\maketitle

\section{Introduction}

This paper is part of an ongoing programme to develop non-commutative Stone dualities \cite{Law6, Law7, Law8, LL}
linking inverse semigroups, topological groupoids and $C^{\ast}$-algebras.
Our goal here is to generalize \cite{Law22} from polycyclic inverse monoids to graph inverse monoids.
Recall that the polycyclic monoid $P_{n}$, where $n \geq 2$, first introduced by Nivat and Perrot \cite{Nivat},
is defined as a monoid with zero by the following presentation
$$P_{n} = \langle a_{1}, \ldots, a_{n}, a_{1}^{-1}, \ldots, a_{n}^{-1}
\colon \: a_{i}^{-1}a_{i} = 1 \,\mbox{and}\, a_{i}^{-1}a_{j} = 0, i \neq j \rangle.$$
It can easily be proved that every non-zero element of $P_{n}$ is of the form $yx^{-1}$ where
$x,y \in A_{n}^{\ast}$, the free monoid on the set $\{a_{1}, \ldots, a_{n} \}$
where the identity is $\varepsilon \varepsilon^{-1}$; here $\varepsilon$ is the empty string.
The product of two elements $yx^{-1}$ and $vu^{-1}$ is zero unless
$x$ and $v$ are {\em prefix-comparable}, meaning that one is a prefix of the other.
If they are prefix-comparable then
$$yx^{-1} \cdot vu^{-1} = \left\{
\begin{array}{ll}
yzu^{-1}   & \mbox{if $v = xz$ for some string $z$}\\
y(uz)^{-1} & \mbox{if  $x = vz$ for some string $z$}
\end{array}
\right.
$$
The polycyclic inverse monoids were independently introduced by Cuntz \cite{C}
and for this reason they are often referred to as {\em Cuntz (inverse) semigroups} in the $C^{\ast}$-algebra literature, such as page~2 of \cite{Pat1}.
However, we shall not use this terminology here: for us, the Cuntz inverse semigroup will be a different semigroup.
Polycyclic inverse monoids are also implicit in the work of Leavitt \cite{L1,L2,L3} on rings without invariant basis number.
There are clear parallels between the Cuntz algebras  and the Leavitt algebras observed in \cite{AP} and \cite{Tomforde}.
Graph inverse semigroups generalize the polycyclics in the following way.
The polycyclic inverse monoid $P_{n}$ is constructed from the free monoid on an $n$-letter alphabet.
Such a monoid can be viewed as the free category on the directed graph consisting of one vertex and $n$ loops.
Replacing free monoids by free categories, we get graph inverse semigroups first introduced by Ash and Hall \cite{AH} in 1975.
In \cite{CK}, Cuntz and Krieger introduced a class of $C^{\ast}$-algebras, constructed from suitable directed graphs, now known as {\em Cuntz-Krieger algebras}.
In 2005, Abrams and Pini \cite{AP} introduced what they called {\em Leavitt path algebras} as the algebra analogues of the Cuntz-Krieger algebras.
These are also the subject of \cite{Tomforde}.
The connection between graph inverse semigroups and the groupoids associated with Cuntz-Krieger algebras is spelled out by Paterson \cite{Pat2}
and, significantly for this paper, in the work by Lenz \cite{Lenz}.
In this paper, we prove three main theorems:
first, we characterize graph inverse semigroups;
second, we describe an order-theoretic completion of such semigroups to what we call Cuntz-Krieger semigroups;
third, we prove that under non-commutative Stone duality, Cuntz-Krieger semigroups may be associated with the \'etale topological groupoid constructed from
the Cuntz-Krieger $C^{\ast}$-algebra constructed from the original directed graph.
Cuntz-Krieger semigroups are equipped with a partially defined join operation
which enables them to be described in terms of {\em Cuntz-Krieger relations} $e = \sum_{f' \in \hat{e}} f'$. 
There are therefore three algebraic structures that arise in studying algebras, in the most general sense,
constructed from directed graphs:
Cuntz-Krieger semigroups,  Leavitt path algebras,  and  graph $C^{\ast}$-algebras.
Our theory also shows that  Cuntz-Krieger semigroups can be used to construct,
and may be constructed from, 
the topological groupoids usually associated with the graph $C^{\ast}$-algebras.

We shall assume the reader is familiar with the rudiments of semigroup theory as described in the early chapters of \cite{H}; 
in particular, Green's relations $\mathscr{H}, \mathscr{L}, \mathscr{R}, \mathscr{D}, \mathscr{J}$.
We note that a congruence is {\em $0$-restricted} if the congruence class of $0$ is just $\{0\}$;
it is {\em idempotent pure} if the only elements related to an idempotent are themselves idempotents.
For an introduction to inverse semigroup theory, Howie \cite{H} is a good starting place but we generally refer to \cite{Law1}
for more advanced theory. 
Throughout this paper, all inverse semigroups will be assumed to have a zero and $\leq$ will denote their natural partial order.
The set of idempotents of an inverse semigroup $S$ is denoted $E(S)$ and is a commutative idempotent subsemigroup.
Inverse semigroups are posets with respect to their natural partial orders.
A poset with zero $X$ is said to be {\em unambiguous}\footnote{Strictly speaking `unambiguous except at zero' but that is too much of a mouthful.}
if for all $x,y \in X$ whenever there exists $z \leq x,y$, where $z$ is nonzero, then either $x \leq y$ or $y \leq x$.
This concept, due to Birget \cite{Birget}, plays an important role in this paper; see also \cite{Hughes, Law12, Law4, Law5}.

\section{The characterization theorem}

There are a number of results in semigroup theory  \cite{Clifford, Leech, Law11, Exel2}  where a left cancellative structure is used to construct an inverse semigroup.
We shall characterize graph inverse semigroups in the spirit of these results but use yet another procedure which harks back to Clifford \cite{Clifford}.
We first state the main result proved in this section.
An inverse semigroup $S$ is said to be a {\em Perrot inverse semigroup} if it satisfies the following conditions:
\begin{description}

\item[{\rm (P1)}] The semilattice of idempotents of $S$ is {\em unambiguous}.

\item[{\rm (P2)}] The semilattice of idempotents of $S$  satisfies the {\em Dedekind height condition}, meaning that for each
non-zero idempotent $e$ the set  $\left| e^{\uparrow} \cap E( S ) \right| <  \infty$, where $e^{\uparrow}$ denotes the set of all elements above $e$.

\item[{\rm (P3)}] Each non-zero idempotent $e$ lies beneath a unique maximal idempotent denoted by $e^{\circ}$.

\item[{\rm (P4)}]  Each non-zero $\mathscr{D}$-class contains a maximal idempotent.

\end{description}

We say that a Perrot semigroup is {\em proper} if it satisfies  a stronger version of (P4):
each non-zero $\mathscr{D}$-class contains a {\em unique} maximal idempotent.
An inverse semigroup is {\em combinatorial} if the $\mathscr{H}$-relation is equality.
Our first theorem can now be stated.

\begin{theorem}\label{the: theorem_one}  
The graph inverse semigroups are precisely the combinatorial proper Perrot semigroups. 
\end{theorem}

We shall prove the above theorem in Section~2.2 by restricting some results we prove in Section~2.1.
In Section~2.3, we shall derive some properties of graph inverse semigroups that will be useful in Section~3.

\subsection{A general construction}

Throughout this paper, categories are small and objects are replaced by identities.
The elements of a category $C$ are called {\em arrows} and the set of identities of $C$ is denoted by $C_{o}$.
Each arrow $a$ has a {\em domain}, denoted by $\mathbf{d}(a)$, and a {\em codomain} denoted by $\mathbf{r}(a)$,
both of these are identities and $a = a\mathbf{d}(a) = \mathbf{r}(a)a$.
The convention in this paper is that $\exists ab$ in a category if and only if $\mathbf{d}(a) = \mathbf{r}(b)$.
Thus our arrows will be drawn like this
$f  \stackrel{a}{\leftarrow} e$.
Given identities $e$ and $f$, the set of arrows $eCf$ from $f$ to $e$ is called a {\em hom-set}
and $eCe$ is a monoid called the {\em local monoid at $e$}.
Elements in the same hom-set are said to be {\em parallel}.
A category is said to be {\em inverse} if for each arrow $a$ there is a unique arrow $a^{-1}$ such that $a = aa^{-1}a$ and $a^{-1} = a^{-1}aa^{-1}$.
Classical inverse semigroup theory can easily be extended to inverse categories and this extension will be carried out silently below whenever needed.
An arrow $a$ is {\em invertible} or an {\em isomorphism} if there is an arrow $a^{-1}$ such that $a^{-1}a = \mathbf{d}(a)$ and $aa^{-1} = \mathbf{r}(a)$.
A category in which every arrow is invertible is called a {\em groupoid}.
All groupoids are inverse categories.
The set of invertible elements of a category forms a groupoid.
A category is said to be {\em skeletal} if any isomorphisms belong only to local submonoids.
We define a category to be {\em left cancellative} if $ab = ac$ implies that $b = c$;
{\em right cancellative} categories are defined analogously.
A category of {\em cancellative} if it is both left and right cancellative.
An {\em ideal} $Z$ in a category $C$ is a non-empty subset with the property that if $a \in I$ and $ba$ is defined then $ba \in Z$, and dually.
In a category $C$, we may define {\em principal right ideals} $aC$, {\em principal left ideals} $Ca$, and {\em principal two-sided ideals} $CaC$.
We may therefore also define analogues of Green's relations $\mathscr{H}$, $\mathscr{L}$, $\mathscr{R}$, $\mathscr{D}$, and $\mathscr{J}$
just as in the monoid case.

Let $e$ be an idempotent in a category $C$.
It is said to {\em split} if there is an identity $i$
and arrows $r$ and $s$ such that $e = sr$ and $i = rs$.
Observe that $s = srs$ and $r = rsr$.
It follows that $e \, \mathscr{D} \. i$.
Thus if an idempotent splits it is $\mathscr{D}$-related to an identity.
The converse is also true.

A category $C$ has {\em zeros} if for each pair of identities $(e,f)$ there is a unique arrow $f \stackrel{0^{f}_{e}}{\leftarrow} e$
such that the following conditions hold:
if $i \stackrel{a}{\leftarrow} e \stackrel{0^{e}_{f}}{\leftarrow} f$ then
$a0^{e}_{f} = 0^{i}_{f}$, and dually.
We denote the set of all such arrows by $Z_{C}$ or just $Z$;
it  is an ideal of the category $C$.

A category $C$ is called a {\em Leech category} if it is left cancellative and  for all $a,b \in C$ if $aC \cap bC \neq \emptyset$ 
then there exists $c \in C$ such that $aC \cap bC = cC$.
A category $I$ is said to be a {\em Clifford category} if it is inverse, has zeros, and every non-zero idempotent is $\mathscr{D}$-related 
to an identity. It follows that in a Clifford category every non-zero idempotent splits.
A Clifford category is said to be {\em proper} if every non-zero idempotent is $\mathscr{D}$-related 
to a {\em unique} identity.
The goal of this section is to prove that (skeletal) Leech categories and (proper) Clifford categories are interdefinable.

Let $C$ be a Leech category.
Put
$$U = \{(a,b) \in C \times C \colon \mathbf{d}(a) = \mathbf{d}(b) \}.$$
Define a relation $\sim$ on $U$ as follows
$$(a,b) \sim (a',b') \Leftrightarrow (a,b) = (a',b')u$$
for some isomorphism $u \in C$.
This is an equivalence relation on $U$ and we denote the equivalence class containing $(a,b)$ by $[a,b]$.
We interpret this as follows.
Suppose that
$$\spreaddiagramrows{2pc}
\spreaddiagramcolumns{2pc}
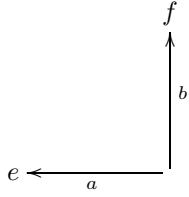
\diagram
& 
&f
\\
&e 
& \lto^{a} \uto_{b}
\enddiagram$$
Then we regard $[a,b]$ as an arrow as follows
$$[e,e] \stackrel{[a,b]}{\longleftarrow} [f,f].$$
We define 
$\mathsf{Inv}(C)$
to be the set of equivalence classes $[a,b]$ together with a family of zero elements $0^{e}_{f}$ where $e,f \in C_{o}$.
Suppose that $[a,b][c,d]$ are potentially composable.
We define their product as follows.
If $bC \cap cC = \emptyset$ then the product is defined to be a zero;
if  $bC \cap cC = dC$ for some $d$ then write $d = bx = cy$ and define the product to be  $[ax,dy]$.
The following diagram should explain exactly what is going on.
$$\spreaddiagramrows{2pc}
\spreaddiagramcolumns{2pc}
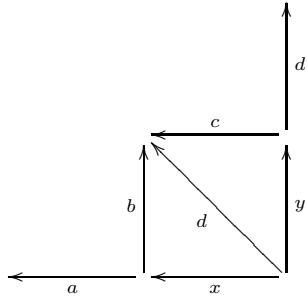
\diagram
& 
& 
&
\\
& 
&
& \lto_{c} \uto_{d}
\\
&
& \lto^{a} \uto^{b}
& \lto^{x} \uto_{y}   \ulto^{d}
\enddiagram$$

\begin{proposition}\label{prop: fred} 
From each Leech category $C$, we may construct a Clifford category $\mathsf{Inv}(C)$ and an embedding $\iota \colon C \rightarrow \mathsf{Inv}(C)$ 
whose inage generates $\mathsf{Inv}(C)$.
If the Leech category $C$ is skeletal, then the Clifford category is proper.
\end{proposition}
\begin{proof} It is straightforward to check that $\mathsf{Inv}(C)$ is a category in which the identities are those elements of the form $[e,e]$ where $e$ is an identity of $C$.
It can be checked that this category is inverse; 
the non-zero idempotents are the elements of the form $[a,a]$ and for each each $[a,b]$ we have that $[a,b]^{-1} = [b,a]$.
If $[a,b]$ and $[c,d]$ are parallel we have that $[a,b] \leq [c,d]$ if and only if $(a,b) = (c,d)p$ for some $p$.
There is therefore an order isomorphism between the poset of principal right ideals of $C$ and the poset of non-zero idempotents of $\mathsf{Inv}(C)$.
Let $e \stackrel{a}{\rightarrow} f$ be an arrow of $C$.
Then $[a,a] \, \mathscr{D} \, [e,e]$.
We have therefore shown that $\mathsf{Inv}(C)$ is a Clifford category.

Let $e \stackrel{a}{\rightarrow} f$ be an arbitrary element of $C$.
Then $[a,e]$ is an element of $\mathsf{Inv}(C)$ from $[e,e]$ to $[f,f]$.
Let $i \stackrel{b}{\rightarrow} e$ be any element of $C$.
Then $[a,e][b,i] = [ab,i]$,
It is easy to check that we have an embedding $\iota \colon C \rightarrow \mathsf{Inv}(C)$.
Observe that if $[a,b]$ is a nonzero element and $i$ is the domain idempotent of both $a$ and $b$ then $[a,b] = [a,i][b,i]^{-1}$.
Thus $\mathsf{Inv}(C)$ is generated as an inverse category by the image of $\iota$.

Suppose now that the Leech category $C$ is skeletal.
Let $[e,e] \, \mathscr{D} \, [f,f]$ where $[e,e]$ and $[f,f]$ are identities.
Then there is $[a,b]$ such that $[a,a] = [e,e]$ and $[b,b] = [f,f]$.
Let $a = eu$ and $b = fv$ where $u$ and $v$ are isomorphisms.
Then $uv^{-1}$ is an isomorphism from $e$ to $f$.
By assumption, $e = f$. 
\end{proof}

We say that the Clifford category $\mathsf{Inv}(C)$ is {\em associated} with the Leech category $C$.

Now let $I$ be  a Clifford category.
Denote by $\mathsf{L}(I)$ the set of all non-zero $a \in I$ such that $a^{-1}a$ is an identity.
That is, all the elements that are $\mathscr{L}$-related to an identity.
Define a partial binary operation on  $\mathsf{L}(I)$ as follows:
$\exists a \cdot b$ if and only if $aa^{-1} \leq b^{-1}b $ in which case $a \cdot b = ab$.

\begin{proposition}\label{prop: daisy} 
From each Clifford category $I$, we may construct a Leech category  $(\mathsf{L}(I), \cdot)$.
If the Clifford category $I$ is proper, then the Leech category is skeletal.
\end{proposition}
\begin{proof} Suppose that $a \cdot b = a \cdot c$.
then $ab = ac$ and so $a^{-1}ab = a^{-1}ac$.
It follows that $b = c$.
It is straightforward to check that  $\mathsf{L}(I)$ is a left cancellative category.
Suppose that $a \cdot \mathsf{L}(I) \cap b \cdot \mathsf{L}(I) \neq \emptyset$.
Then $aI \cap bI \neq \emptyset$ and so  the idempotent $aa^{-1}bb^{-1}$ is non-zero in $I$.
By assumption, there exists $w \in I$ such that $w^{-1}w$ is an identity and $ww^{-1} = aa^{-1}bb^{-1}$ 
We claim that
$$a \mathsf{L}(I) \cap b \mathsf{L}(I) = w \mathsf{L}(I).$$
Let  $u = ac = bd$ be any element of the lefthand side.
We have that 
$$u = ac = bb^{-1}ac = bb^{-1}aa^{-1} ac = ww^{-1}ac.$$
Thus $u = w(w^{-1}ac)$.
Observe that 
$$(w^{-1}ac)^{-1}w^{-1}ac = (b^{-1}ac)^{-1}b^{-1}ac$$
and $b^{-1}ac = b^{-1}bd = d$.
It follows that $w^{-1}ac \in  \mathsf{L}(I)$.
Now let $wx$ be any element in the righthand side.
Then $wx = a(a^{-1}bb^{-1}wx)$.
It is easy to check that $a^{-1}bb^{-1}wx \in \mathsf{L}(I)$. 
We have therefore proved that  $(\mathsf{L}(I), \cdot)$ is a Leech category.

Suppose that the Clifford category $I$ has the property that each non-zero idempotent is $\mathscr{D}$-related to a unique identity.
Let $e$ and $f$ be identities of $\mathsf{L}(I)$.
These are just identities of $I$.
Suppose that $e$ and $f$ are isomorphic in  $\mathsf{L}(I)$.
Then they are $\mathscr{D}$-related in $I$ and so must be equal.
\end{proof}

We say that the Leech category  $\mathsf{L}(I)$  is {\em associated} with the Clifford category $I$.

We may now establish a correspondence between Leech categories and Clifford categories.

\begin{proposition} \mbox{}
\begin{enumerate}

\item Each Clifford category is isomorphic to the Clifford category constructed from its associated Leech category.

\item Each Leech category is isomorphic to the Leech category constructed from its associated Clifford category.

\item Under these correspondences skeletal Leech categories correspond to proper Clifford categories. 

\end{enumerate}
\end{proposition}
\begin{proof} (1). Let $I$ be a Clifford category.
Let $s \in I$ be a non-zero element.
By assumption, there is an identity $i$, not necessarily unique, such that $s \, \mathscr{D} \, i$.
It follows that we may find elements $a,b \in I$ such that $s = ab^{-1}$ and $bb^{-1} = a^{-1}a = i$
and $aa^{-1} = ss^{-1}$ and $bb^{-1} = s^{-1}s$.
The diagram below shows what is going on
$$\spreaddiagramrows{2pc}
\spreaddiagramcolumns{2pc}
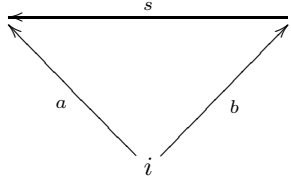
\diagram
& 
&
& \llto_{s}
\\
& 
& i \ulto^{a} \urto_{b}
&
\enddiagram$$
It follows that $[a,b] \in  \mathsf{Inv}(\mathsf{L}(S))$.
Suppose now that there is an identity $j$ and elements $a_{1}$ and $b_{1}$ as above and such that $a = a_{1}b_{1}^{-1}$.
However, $i \, \mathscr{D} \, j$ and $i$ and $j$ are both identities.
There is therefore an isomorphism $j \stackrel{g}{\rightarrow} i$ such that $a_{1} = ag$ and $b_{1} = bg$.
We have therefore proved that 
there is a map $s \mapsto [a,b]$ from the non-zero elements of $S$ to the non-zero elements of $\mathsf{Inv}(\mathsf{L}(S))$.
Suppose that $s$ maps to $[a,b]$ and $t$ maps to $[c,d]$ and that $[a,b] = [c,d]$.
Then $(a,b) = (c,d)u$ for some isomorphism $u$.
Then $s = ab^{-1} = (cu)(du)^{-1} = cd^{-1} = t$.
Thus our map is injective.
It is clearly surjective.
Suppose that $st$ is defined and non-zero.
Let $s = ab^{-1}$ where $a$ and $b$ start at the identity $i$,
and $t = cd^{-1}$ where $c$ and $d$ start at the identity $j$.
In particular, $bb^{-1}cc^{1}$ is a non-zero idempotent.
Let $w \in I$ such that $w^{-1}w$ is an identity and $ww^{-1} = bb^{-1}cc^{1}$ 
Put $x = b^{-1}cc^{-1}$ and $y = c^{-1}bb^{-1}$.
Then $bx = cy$.
Observe that $st = (ax)(dy)^{-1}$.
It is now routine to check that we have defined an isomorphism from $I$ to  $\mathsf{Inv}(\mathsf{L}(S))$.

(2). Let $C$ be a Leech category.
Define a map from $C$ to $\mathsf{L}(\mathsf{Inv}(C))$ by $a \mapsto [a,e]$ where $e$ is the domain of $a$.
Then this is an injective functor.
Let $[a,b] \in \mathsf{Inv}(C)$ be such that $[a,b]^{-1}[a,b]$ is an identity.
Then $[b,b]$ is an identity and so $[b,b] = [e,e]$ for some identity $e$.
Thus $b = eu$ where $u$ is an isomorphism.
Thus $[a,b] = [a,eu] = [au^{-1},e] = [a',e]$.
Suppose that $[a,b] = [c,f]$ where $f$ is an identity.
Then $(c,f) = (a',e)u$ where $u$ is an isomorphism.
But $f = eu$.
Thus $e = f$ and $u$ is just an identity.
It follows that $c = a'$.
We have therefore also shown that our map above is surjective.

(3). This is immediate by (1) and (2) above and Propositions~\ref{prop: fred} and \ref{prop: daisy}.
\end{proof}

It will be useful to consider how some important properties are translated under the above correspondence.
An inverse semigroup or inverse category is said to be {\em $E^{\ast}$-unitary} if every element above a non-zero idempotent is itself an idempotent.

\begin{lemma}\label{le: dora} Let $C$ be a Leech category
\begin{enumerate}

\item The category $C$ is also right cancellative if and only if $\mathsf{Inv}(C)$ is $E^{\ast}$-unitary.

\item The local groups of $C$ are trivial if and only if $\mathsf{Inv}(C)$ is combinatorial.

\end{enumerate}
\end{lemma}
\begin{proof} (1). Suppose that $C$ is also right cancellative.
Let $[a,a] \leq [b,c]$.
Then $(a,a) = (b,c)p$.
Thus $bp = cp$.
By  right cancellativity, we have that $b = c$, as required.
Suppose now that $\mathsf{Inv}(C)$ is $E^{\ast}$-unitary.
Let $ab = cb = d$.
Then $[d,d] \leq [a,c]$.
By assumption, $[a,c]$ is an idempotent and so $a = c$, as required. 

(2). Suppose that the local groups of $C$ are trivial.
Let $[a,b] \, \mathscr{H} \, [c,d]$.
Then $a = cu$ and $b = dv$ where $u$ and $v$ are invertible.
It follows that $uv^{-1}$ is an invertible element in a local group ans so trivial.
Thus $u = v$ and so $[a,b] = [c,d]$.
Suppose now that  $\mathsf{Inv}(C)$ is combinatorial.
Let $u$ be an invertible element in the local group based at the identity $e$.
Then $[u,e] \, \mathscr{H} \, [e,e]$.
By assumption, we have that $[u,e] = [e,e]$.
Then $(u,e) = (e,e)v$ where $v$ is invertible.
It follows that $v$ is an identity and $u = e$, as required.
\end{proof}

\begin{remark}{\em 
When a Leech category $C$ has only trivial invertible elements, 
the equivalence class $[a,b]$ consists only of the element $(a,b)$.
It is convenient in this case to denote $[a,b]$ by $ab^{-1}$.
We shall do this in Section~2.2.}
\end{remark}

Our results so far have been couched in the language of inverse categories,
but it is a simple matter to convert them into the language of inverse semigroups.
We describe two constructions that do just that.

Let $I$ be an inverse category with zeros.
Denote the set of zeros by $Z$.
Adjoin a new symbol $0$ to $I$.
The resulting set $I^{0}$ becomes a semigroup when all undefined products in $I$ are defined to be equal to zero 
and the element $0$ functions as a zero.
The set $Z \cup \{0\}$ is a semigroup ideal of $I^{0}$.
The semigroup obtained by taking the Rees quotient of $I^{0}$ by $Z \cup \{0\}$ is denoted by
$I^{c}$ and called the {\em contraction} of $I$.

We say that an inverse semigroup $S$ with zero has {\em maximal idempotents} if for each non-zero idempotent $e$
there is a unique maximal idempotent $e^{\circ}$ such that $e \leq e^{\circ}$.
Let $S$ be an inverse semigroup with zero and maximal idempotents.
We consider all triples of the form $((xx^{-1})^{\circ},x,(x^{-1}x)^{\circ})$ where $x$ is non-zero
together with all triples of the form $(e,0,f)$ where $e$ and $f$ are arbitrary maximal idempotents.
We define the product $(e,x,f)(f,y,i) = (e,xy,i)$ and all other products are undefined.
We denote the resulting category by $S^{e}$ and call it the {\em expansion} of $S$
The proofs of the following are routine.

\begin{lemma} \mbox{}
\begin{enumerate}

\item For each inverse category with zeros $I$, we have that $I$ is isomorphic to $I^{ce}$.

\item For each inverse semigroup with zero and maximal idempotents $S$, we have that $S$ is isomorphic to $S^{ec}$.

\end{enumerate}

\end{lemma}

Let $C$ be a Leech category.
Put $\mathsf{S}(C) = \mathsf{Inv}(C)^{c}$ and call it the 
{\em inverse semigroup associated with the Leech category $C$.}
We may now translate the results of this section into inverse semigroup-theoretic terms.
Part (2) below uses the second part of Lemma~\ref{le: dora}.

\begin{theorem}\label{the: jack} There is bijective correspondence between Leech categories and inverse semigroups 
with zero having maximal idempotents and in which each non-zero idempotent is
$\mathscr{D}$-related to a maximal idempotent.
\begin{enumerate}

\item In the above correspondence, skeletal Leech categories correspond to the inverse semigroups in which each non-zero idempotent is
$\mathscr{D}$-related to a unique maximal idempotent.

\item In the above correspondence, Leech categories with trivial groupoids of isomorphisms 
correspond to those inverse semigroups which are combinatorial  and in which each non-zero idempotent is $\mathscr{D}$-related to a unique maximal idempotent.

\end{enumerate}
\end{theorem}

The above theorem deals with conditions (P3) and (P4) in the definition of Perrot semigroups.
In the next section, we shall deal with conditions (P1) and (P2).

\subsection{Proof of the characterization}

A category $C$ is said to be {\em right rigid} if $aC \cap bC \neq \emptyset$ implies that $aC \subseteq bC$ or $bC \subseteq aC$.
A {\em left Rees category} is a  left cancellative, right rigid category in which each principal right ideal is contained in only a finite number of distinct principal right ideals
in addition, we shall also require that there are non-invertible elements.
Left Rees categories are certainly Leech categories.
We may therefore ask for a characterization of their associated inverse semigroups. 

\begin{theorem}\label{the: types_leech} \mbox{}
\begin{enumerate}

\item Left Rees categories correspond to Perrot semigroups

\item Skeletal left Rees categories correspond to proper Perrot semigroups.

\item Left Rees categories with trivial groupoids of invertible elements correspond to combinatorial proper Perrot semigroups.

\end{enumerate}
\end{theorem} 
\begin{proof} The proof of these is immediate from Theorem~\ref{the: jack} and Lemma~\ref{le: dora} and the description of the poset of idempotents
in Proposition~\ref{prop: fred} in terms of the poset of principal right ideals of the Leech category.
\end{proof}

By virtue of the third part of the above theorem, it only remains to relate left Rees categories with trivial groupoids of invertible elements to free categories, which we now do.

A \emph{directed graph} $G$ is a collection of \emph{vertices} $G_0$ and a collection of \emph{edges} $G_1$ 
together with two functions $\dom , \ran  \colon G_1 \rightarrow G_0$ called the\emph{ domain} or {\em source} and the \emph{range} or {\em target}, respectively.
The {\em in-degree} of a vertex $v$ is the number of edges $x$ such that $\mathbf{r} (x) = v$
and the {\em out-degree} of a vertex $v$ is the number of edges $x$ such that $\mathbf{d} (x) = v$.
A {\em sink} is a vertex whose out-degree is zero and a {\em source} is a vertex whose in-degree is zero.
Two edges $x$ and $y$ \emph{match} if $\dom (x) = \ran (y)$.
A \emph{path} is any sequence of edges $x_1 \ldots x_n$ such that $x_i$ and $x_{i+1}$ match for all $i=1, \ldots , n_1$;
such a path looks like this $\stackrel{x_{1}}{\leftarrow} \stackrel{x_{2}}{\leftarrow} \ldots \stackrel{x_{n}}{\leftarrow}$.
The {\em length} $\left| x \right|$ of a path $x$ is the total number of edges in it.
The empty path, or path of length zero, at the vertex $v$ is denoted by $1_{v}$. 
The \emph{free category} $G^{\ast}$ generated by the directed graph $G$ is the set of all paths equipped with concatenation as the partial binary operation.
If $x,y \in G^{\ast}$ are such that either $x = yz$ or $y = xz$ for some path $z$ then we say that $x$ and $y$ are {\em prefix-comparable}.
The proof of the following is straightforward.

\begin{proposition}\label{prop: one_direction} 
Free categories are left Rees categories in which the groupoid of invertible elements is trivial.
\end{proposition}

Because of the above result,  given a directed graph $G$, 
we define the \emph{graph inverse semigroup} $P_{G}$ to be the inverse semigroup $\mathsf{Inv}(G^{\ast})^{c}$ using the notation from Section~2.1.
The free category has no non-trivial invertible elements and so each equivalence class is denoted by $xy^{-1}$ using Remark~2.6.
Thus the non-zero elements of $P_{G}$ are of the form $uv^{-1}$ where $u,v$ are paths in $\mathcal{G}$ with common domain.
With elements in this form, multiplication assumes the following shape:
\begin{equation*}
xy^{-1} \cdot uv^{-1} =
\begin{cases}
 xzv^{-1} & \mbox{if $u=yz$ for some path $z$}\\
 x \left( vz \right)^{-1} & \mbox{if $y=uz$ for some path $z$}\\
 0 & \mbox{otherwise.}\\
\end{cases}
\end{equation*}
Let $xy^{-1}$ and $uv^{-1}$ be non-zero elements of $P_{G}$. 
Then
$$xy^{-1} \leq uv^{-1}
\Leftrightarrow 
\exists p \in \mathcal{G}^{\ast} \text{ such that }  
x = up \text{ and } y = vp.$$
If $xy^{-1} \leq uv^{-1}$ or $uv^{-1} \leq xy^{-1}$ then we say $xy^{-1}$ and $uv^{-1}$ are \emph{comparable}.

It only remains to characterize free categories.
First we recall some results that were proved in a much more general frame in \cite{Law5}.
Because of the importance of the main result proved below, we felt it important to include all the details in this case.
Let $C$ be a category.
A non-invertible element $a \in C$ is said to be an {\em atom} iff $a = bc$ implies that either $b$ or $c$ is invertible.
A principal right ideal $aC$ is said to be {\em submaximal} if $aC \neq \ran(a)C$ and there are no proper principal right ideals between $aC$ and $\ran(a)S$.
The proofs of the following are routine.

\begin{lemma}\label{le: theta} Let $C$ be a left cancellative category. 
\begin{enumerate}
\item If $e = xy$ is an identity then $x$ is invertible with inverse $y$.
In addition,  the set of invertible elements is trivial iff for all identities $e$ we have that $e = xy$ implies that $x$ and $y$ are identities.
\item We have that $aC = bC$ iff $a = bg$ where $g$ is an invertible element.
\item $aC = eC$ for some identity $e$ iff $a$ is invertible.
\item The maximal principal right ideals are those generated by identities.
\item The element $a$ is an atom iff $aC$ is submaximal.
\end{enumerate}
\end{lemma}

The following theorem was inspired by similar characterizations of free monoids; see pp~103--104 of \cite{Lallement}.

\begin{theorem}\label{the: free}
A category is free if and only if it is a left Rees category having a trivial groupoid of invertible elements.
\end{theorem} 
\begin{proof} By Proposition~\ref{prop: one_direction},  only one direction needs proving.
Let $C$ be a left Rees category having a trivial groupoid of invertible elements.
We prove that it is isomorphic to a free category generated by a directed graph.
By assumption, there is at least one non-invertible element $a \in C$.
The principal right ideal $aC$ is therefore not maximal by parts (3) and (4) of Lemma~\ref{le: theta}.
By assumption, the poset of principal right ideals between $aC$ and $\mathbf{r}(a)C$ is linearly ordered and finite.
It follows that submaximal principal right ideals exist and so atoms exit by part (5) of Lemma~\ref{le: free}.
We have also proved that every non-maximal principal right ideal is contained in a submaximal principal right ideal.
In fact, it must be contained in a unique such ideal by right rigidity.
Let $X$ be a transversal of generators of the set of submaximal principal right ideals.
We may regard $X$ as a directed graph: the set of vertices is $C_{o}$ and if $a \in X$ then $\ran(a) \stackrel{a}{\longleftarrow} \dom(a)$.
We shall prove that the subcategory $X^{\ast}$ generated by $X$ is isomorphic to $C$ and that it is free.
Let $a \in C$ be a non-identity element.
If $aC$ is submaximal then $a$ is an atom and since we are assuming that the invertible elements are trivial it follows
by part (2) of  Lemma~\ref{le: theta} that $a \in X$.
Suppose, therefore, that $aC$ is not submaximal.
Then $aC \subseteq a_{1}C$ where $a_{1} \in X$.
Thus $a = a_{1}b_{1}$.
We now repeat this argument with $b_{1}$.
We have that $b_{1}C \subseteq a_{2}C$ where $a_{2} \in X$.
Thus $b_{1} = a_{2}b_{2}$.
Observe that $aC \subseteq a_{1}C \subseteq a_{1}a_{2}C$.
Continuing in this way, 
and using the fact that each principal right ideal is contained in only a finite set of principal right ideals,
we have shown that $a = a_{1} \ldots a_{n}$ where $a_{i} \in X$.
We have therefore proved that $C = X^{\ast}$.
It remains to show that each element of $C$ can be written uniquely as an element of $X^{\ast}$.
Suppose that
$$a = a_{1} \ldots a_{m} = b_{1} \ldots b_{n}$$
where $a_{i},b_{j} \in X$.
Then $a_{1}C \cap b_{1}C \neq \emptyset$.
But both principal right ideals are submaximal and so $a_{1}C = b_{1}C$.
Hence $a_{1} = b_{1}$.
By left cancellation we get that 
$$a_{2} \ldots a_{m} = b_{2} \ldots b_{n}.$$
If $m = n$ then $a_{i} = b_{i}$ for all $i$ and we are done.
If $m \neq n$ then we deduce that a product of atoms is equal to an identity
but this contradicts part (5) of Lemma~\ref{le: theta}.
\end{proof}
 
The proof of Theorem~\ref{the: theorem_one} now follows from Theorem~\ref{the: free} and the third part of Theorem~\ref{the: types_leech}.

\subsection{Some properties of graph inverse semigroups}

In this section,  we shall prove a few results about graph inverse semigroups needed in Section~3.
In an inverse semigroup $S$, 
elements $s$ and $t$ are said to be {\em compatible} if both $s^{-1}t$ and $st^{-1}$ are idempotents;
they are said to be  {\em orthogonal} if $s^{-1}t = 0 = st^{-1}$.
We denote by $s^{\downarrow}$ the set of all elements below $s$.
We often write $\mathbf{d}(s) = s^{-1}s$ and $\mathbf{r}(s) = ss^{-1}$ in inverse semigroups.
The proofs of the following may be found in \cite{Law1}.

\begin{lemma}\label{le: basics} Let $S$ be an inverse semigroup.
\begin{enumerate}

\item For each element $s$ in an inverse semigroup $S$ the subset $s^{\downarrow}$ is compatible.

\item If $s$ and $t$ are compatible then $s \wedge t$ exists and $\mathbf{d}(s \wedge t) = \mathbf{d}(s)\mathbf{d}(t)$, and dually.

\item If $s$ and $t$ are compatible and $\mathbf{d}(s) \leq \mathbf{d}(t)$ then $s \leq t$, and dually.

\item If $s \wedge t$ exists then $as \wedge at$ exists for any $a$ and $a(s \wedge t) = as \wedge at$, and dually.

\end{enumerate}
\end{lemma}

The following is Remark~2.3 of \cite{Lenz}.

\begin{lemma}\label{le: jinxy} 
If $S$ is an $E^{\ast}$-unitary inverse monoid then $(S,\leq)$ is a meet semilattice.
\end{lemma}

The proof of the following result follows from part (1) of Lemma~\ref{le: dora} and the fact that free categories are cancellative.

\begin{proposition}\label{prop: lenz} 
Combinatorial proper Perrot semigroups are $E^{\ast}$-unitary.
\end{proposition}

An inverse semigroup is called an {\em inverse $\wedge$-semigroup} if each pair of elements has a meet.
Since graph inverse semigroups are $E^{\ast}$-unitary by Proposition~\ref{prop: lenz} the proof of the following is immediate
by Lemma~\ref{le: jinxy}.

\begin{corollary}\label{cor: wedge} 
Graph inverse semigroups are  inverse $\wedge$-semigroups.
\end{corollary}

Unambiguity is an important feature of graph inverse semigroups.

\begin{lemma} Let $S$ be an inverse semigroup whose poset of idempotents is unambiguous.
Then the partially ordered set $(S,\leq)$ is unambiguous if and only if $S$ is $E^{\ast}$-unitary. 
\end{lemma}
\begin{proof} Let $S$ be an $E^{\ast}$-unitary inverse semigroup.
Let $0 \neq a \wedge b \leq a,b$.
Then $0 \neq \mathbf{d}(a \wedge b) \leq \mathbf{d}(a), \mathbf{d}(b)$.
By unambiguity, it follows that either $\mathbf{d}(a) \leq \mathbf{d}(b)$ or vice-versa.
We assume the former without loss of generality.
Thus $\mathbf{d} \leq \mathbf{d}(b)$.
However $a^{-1}b$ and $ab^{-1}$ are both above non-zero idempotents.
Thus from the fact that the semigroup is $E^{\ast}$-unitary we have that $a$ is compatible with $b$.
By Lemma~\ref{le: basics}, we have that $a \leq b$, as required

Let $(S,\leq)$ be an unambiguous poset. 
We prove that $S$ is $E^{\ast}$-unitary.
Let $0 \neq e \leq s$ where $e$ is an idempotent.
We prove that $s$ is an idempotent.
Clearly $e \leq s^{-1}$.
Thus $s$ and $s^{-1}$ are comparable.
If $s \leq s^{-1}$ then by taking inverses we also have that $s^{-1} \leq s$ and vice-versa.
It follows that $s = s^{-1}$.
Thus $s^{2} = ss^{-1}$ is an idempotent.
Now $s$ and $s^{2}$ are also comparable.
If $s \leq s^{2}$ then $s$ is an idempotent and we are done.
If $s^{2} \leq s$ then $s = ss^{-1}s = s^{3} \leq s^{2}$ and so $s \leq s^{2}$ and $s$ is again an idempotent.
\end{proof}

The proof of the following is now immediate.

\begin{corollary}\label{cor: unambiguous} 
The natural partial order on a graph inverse semigroup is unambiguous.
\end{corollary}

An inverse semigroup is said to be {\em $F^{\ast}$-inverse} if it is $E^{\ast}$-unitary and every non-zero element lies beneath 
a unique maximal element. 

\begin{lemma} 
Every graph inverse semigroup is $F^{\ast}$-inverse.
\end{lemma}
\begin{proof} Let $a$ be any non-zero element in a graph inverse semigroup $S$.
By assumption, the set $\mathbf{d}(a)^{\uparrow} \cap E(S)$ is finite and linearly ordered.
There is a map from $a^{\uparrow}$ to $\mathbf{d}(a)^{\uparrow} \cap E(S)$ given by $b \mapsto \mathbf{d}(b)$.
Let $e$ be the maximum element in the image of this map.
Suppose that $a \leq b,c$ are such that $\mathbf{d}(b) = \mathbf{d}(c) = e$.
Both $b^{-1}c$ and $bc^{-1}$ are idempotents and so $b \sim c$.
It follows by Lemma~ \ref{le: basics} that $b = c$.
Let $a \leq b$ where $\mathbf{d}(b) = e$.
Let $a \leq c$.
But $S$ is unambiguous and so $c < b$ or $b \leq c$.
The latter cannot hold because we would then have $e = \mathbf{d}(b) = \mathbf{d}(c)$ and so $b = c$ as above.
It follows that $c < b$ and so $b$ is the unique maximal element above $a$.
\end{proof}

Let $E$ be a poset with zero.
Given $e,f \in E$ we say that $e$ {\em immediately covers} $f$ if $e > f$ and there is no $g \in E$ such that $e > g > f$.
The elements of $E$ that cover zero are said to be {\em $0$-minimal}.
For each $e \in E$ define $\hat{e}$ to be the set of elements of $E$ that are immediately covered by $e$.
The poset is said to be {\em pseudofinite} if whenever $e > f$ there exists $g \in \hat{e}$
such that $e > g \geq f$, and for which the sets $\hat{e}$ are always finite.
Finally, the poset is said to be \emph{$0$-disjunctive} 
if for each $0 \neq f \in E$ and $e$ such that $0 \neq e < f$, there exists $0 \neq e' <f$ such that $e \wedge e' = 0$.

\begin{lemma}\label{le: compactness} Let $P_{G}$ be a graph inverse semigroup.
\begin{enumerate}

\item The inverse semigroup $P_{G}$ has no $0$-minimal idempotents if and only if the in-degree of each vertex is at least one.
 
\item  The inverse semigroup $P_{G}$ has a $0$-disjunctive semilattice of idempotents if and only if the in-degree of each vertex is either 0 or at least 2.

\item The semilattice of idempotents of $P_{G}$ is pseudofinite if and only if the in-degree of each vertex is finite.

\end{enumerate}
\end{lemma}
\begin{proof}

(1) Let $e$ be a vertex with in-degree at least 1, and let $b$ be an edge with target $e$.
Let $x$ be a path with source $e$, where we include the possibility that $x$ is the empty path at $e$.
Then $xb(xb)^{-1} \leq xx^{-1}$.
It follows that if the in-degree of each vertex is at least 1 then there can be no 0-minimal idempotents.
Now let $e$ be a vertex with in-degree 0.
Then $1_{e}1_{e}^{-1}$ is a 0-minimal idempotent.
 
(2) Suppose that $E$ is $0$-disjunctive.
Let $v$ be any vertex.
Let $x$ be any path that starts at $v$ including the empty path $1_{v}$.
Suppose that the in-degree of $v$ is not zero.
Then there is at least one edge $w$ into $v$.
It follows that $xw(xw)^{-1} \leq xx^{-1}$.
By assumption, there exists $zz^{-1} \leq xx^{-1}$ such that $zz^{-1}$ and $xw(xw)^{-1}$ are orthogonal.
Now $z = xp$ for some non-empty path $p$.
It follows that $w$ is not a prefix of $p$ and so there is at least one other edge coming into the vertex $v$.
 
Suppose now that the in-degree of each vertex is either zero or at least two.
Let $yy^{-1} < xx^{-1}$ where $y  = xp$ where the target of $x$ is the vertex $v$.
Since $p$ is a non-empty path that starts at $v$ it follows that there is at least one other edge $w$ with target $v$
that differs from the first edge of $p$.
Thus $xw(xw)^{-1} \leq xx^{-1}$ and $xw(xw)^{-1}$ and $yy^{-1}$ are orthogonal.

(3) Straightforward.
\end{proof}

An inverse semigroup is said to be {\em distributive} if it has joins of all finite non-empty compatible subsets
and multiplication distributes over those finite joins that exist.
{\em Homomorphisms of distributive inverse semigroups} are those which preserve compatible jojns.
An inverse semigroup is said to be {\em Boolean } if it is distributive and its semilattice of idempotents is a generalized Boolean algebra.
Joins of orthogonal sets are called {\em orthogonal joins}.
An inverse semigroup is said to be {\em orthogonally complete} if it has all non-empty finite orthogonal joins
and multiplication distributes over any such joins that exist.
The following result connects the approach adopted in the special case \cite{Law22} with the one adopted here.

\begin{proposition}\label{prop: important}  Let $S$ be an inverse semigroup with zero whose natural partial order is unambiguous.
Then $S$ is distributive if and only if it is orthogonally complete.
\end{proposition}
\begin{proof} Only one direction needs proving.
Let $\{s_{1}, \ldots, s_{m}\}$ be a finite non-empty compatible subset of non-zero elements.
Consider $s_{1}$ and $s_{2}$.
Since they are compatible we must have that $s_{1} \wedge s_{2}$ exists by Lemma~\ref{le: basics}.
If $s_{1} \wedge s_{2} = 0$ then $s_{1}$ and $s_{2}$ are orthogonal.
If $s_{1} \wedge s_{2} \neq 0$ then $s_{1}$ and $s_{2}$ are comparable.
We may assume that $s_{1} \leq s_{2}$.
Discard the smaller element.
Continuing in this way, and relabelling if necessary, we obtain a non-empty orthogonal subset $\{s_{r}, \ldots, s_{m} \}$. 
Each of the discarded elements $s_{1}, \ldots, s_{r-1}$ is less than one of the elements in  $\{s_{r}, \ldots, s_{m} \}$.
It follows that
$$\bigvee_{i=1}^{m} s_{i} = \bigvee_{i=r}^{m} s_{i}.$$
Now let $t \in S$.
The set   $\{ts_{r}, \ldots, ts_{m} \}$ is also orthogonal.
Thus  $\bigvee_{i=r}^{m} ts_{i}$ is defined.
A routine argument shows that
$$\bigvee_{i=1}^{m} ts_{i} = \bigvee_{i=r}^{m} ts_{i}.$$
\end{proof}

\section{Completion: the Cuntz-Krieger semigroups}

\noindent
{\bf Note: }Throughout this section, unless otherwise stated, 
$G$ is a directed graph satisfying the condition that the in-degree of each vertex is at least 2 and finite.\\

We define the {\em Cuntz-Krieger inverse semigroup} $CK_{G}$ in the following way:
\begin{description}

\item[{\rm (CK1)}] It is Boolean.

\item[{\rm (CK2)}] It contains a copy of $P_{G}$ and every element of $CK_{G}$ is a join of a finite, non-empty compatible subset of $P_{G}$.

\item[{\rm (CK3)}] $e = \bigvee_{f' \in \hat{e}} f'$ for each maximal idempotent $e$ of $P_{G}$, where $\hat{e}$ is the set of idempotents immediately covered by $e$.

\item[{\rm (CK4)}] It is the freest distributive inverse semigroup satisfying the above conditions.

\end{description}
We shall prove that this inverse semigroup exists and show how to construct it.
In addition, we shall explain how it is related to the Cuntz-Krieger $C^{\ast}$-algebra via the associated topological groupoid
and briefly consider its representation theory.
In the case where $G$ has one vertex and $n$ loops, the graph inverse semigroup is just the polycyclic monoid $P_{n}$
and its completion is $C_{n}$,  the {\em Cuntz semigroup of degree $n$}, constructed in \cite{Law22}.

\subsection{The Lenz arrow relation}

The key concept we shall need in this is the {\em Lenz arrow relation} introduced in \cite{Lenz}.
Let $a,b \in S$ where $S$ is an inverse $\wedge$-semigroup.
We define $a \rightarrow b$ iff for each non-zero element $x \leq a$, we have that $x \wedge b \neq 0$.
Observe that $a \leq b \Rightarrow a \rightarrow b$.
We write $a \leftrightarrow b$ iff $a \rightarrow b$ and $b \rightarrow a$.

\begin{lemma}\label{le: lenz} Let $S$ be a $\wedge$-semigroup.
\begin{enumerate}

\item If $a \rightarrow 0$ then $a = 0$.

\item $\rightarrow$ is reflexive and transitive.

\item $\leftrightarrow$ is a $0$-restricted congruence.

\item Let $a_{1}, a_{2}, b_{1}, b_{2} \neq 0$.
Suppose that $a_{1} \leftrightarrow a_{2}$ and $b_{1} \leftrightarrow b_{2}$.
Then $a_{1} \wedge b_{1} = 0$ if and only if $a_{2} \wedge b_{2}$.
In the case where they are nonzero, we have that
$(a_{1} \wedge b_{1}) \leftrightarrow (a_{2} \wedge b_{2})$.

\item  If $S$ is in addition distributive and  $a_{1} \leftrightarrow a_{2}$ and $b_{1} \leftrightarrow b_{2}$
and $a_{1} \sim b_{1}$ and $a_{2} \sim b_{2}$ then $(a_{1} \vee b_{1}) \leftrightarrow (a_{2} \vee b_{2})$.

\end{enumerate}
\end{lemma}
\begin{proof} The proofs of (1), (2) and (3) are either immediate or can be deduced from \cite{Lenz,Law8}.

(4) Let $0 \neq a_{1} \wedge b_{1}$.
Then, in particular, $0 \neq a_{1} \leq a_{1}$.
By assumption, $0 \neq x \wedge a_{2}$.
But we clearly have $0 \neq x \wedge a_{2} \leq b_{1}$.
by assumption, $0 \neq (x \wedge a_{2}) \wedge b_{2}$.
The proof of the remainder of this part is now straightfoward.

(5) The proof is straightfoward once we observe that in a distributive inverse $\wedge$-semigroup
we have that  if $a \vee b$ exists then for any $c \in S$ we have that $c \wedge (a \vee b) = (c \wedge a) \vee (c \wedge b)$ \cite{Resende}.
\end{proof}

If $a,a_{1}, \ldots, a_{m} \in S$ define $a \rightarrow \{a_{1}, \ldots, a_{m}\}$ iff for each non-zero element $x \leq a$
we have that $x \wedge a_{i} \neq 0$ for some $i$.
Finally, we write 
$$\{a_{1}, \ldots, a_{m}\} \rightarrow \{b_{1}, \ldots, b_{n}\}$$
iff $a_{i} \rightarrow  \{b_{1}, \ldots, b_{n}\}$
for $1 \leq i \leq m$,
and we write
$$\{a_{1}, \ldots, a_{m}\} \leftrightarrow \{b_{1}, \ldots, b_{n}\}$$
iff both 
$\{a_{1}, \ldots, a_{m}\} \rightarrow \{b_{1}, \ldots, b_{n}\}$
and 
$\{b_{1}, \ldots, b_{n}\} \rightarrow \{a_{1}, \ldots, a_{m}\}$. 
A subset $Z \subseteq A$ is said to be a {\em cover} of $A$ if for each $a \in A$ there exists $z \in Z$ such that $a \wedge z \neq 0$.
A special case of this definition is the following.
A finite subset $A \subseteq a^{\downarrow}$ is said to be a {\em (Lenz) cover} of $a$ if $a \rightarrow A$.
The proof of the following is straightforward.

\begin{lemma}\label{le: lenz_cover} Let $S$ be an inverse $\wedge$-semigroup.
Suppose that $a \rightarrow \{a_{1}, \ldots, a_{n}\}$.
Then $\{a \wedge a_{i} \colon 1 \leq i \leq n \}$ is a Lenz cover of $a$.
\end{lemma}

An inverse $\wedge$-semigroup $S$ is said to be {\em separative} if and only if the relation $\leftrightarrow$ is equality.

\begin{proposition}\label{prop: separative} Let $S$ be an unambiguous $E^{\ast}$-unitary inverse semigroup.
Then $S$ is separative if and only if the semilattice of idempotents $E(S)$ is $0$-disjunctive.
\end{proposition}
\begin{proof} 
Suppose first that $S$ is separative.
We prove that $E(S)$ is $0$-disjunctive.
Let $0 \neq e < f$.
Then $e \rightarrow f$.
By assumption, we cannot have that $f \rightarrow e$.
Thus for some $e' \leq f$ we must have that $e' \wedge e = 0$.
It follows that $E(S)$ is $0$-disjunctive.

We shall now prove the converse.
We shall prove that if $s \nleq t$ where $s$ and $t$ are non-zero then there exists $0 \neq s' \leq s$ such that $s' \wedge t = 0$.
Before we do this, we show that this property implies that $S$ is separative.
Suppose that $s \leftrightarrow t$ and that $s \neq t$.
Then we cannot have both $s \leq t$ and $t \leq s$.
Suppose that $s \nleq t$. 
Then we can find $0 \neq s' \leq s$ such that $s' \wedge t = 0$ which contradicts our assumption.

We now prove the claim.
We shall use Corollary~\ref{cor: unambiguous} that tells us that the inverse semigroup itself is an unambiguous poset.
Suppose that $s \wedge t = 0$.
But then $0 \neq s \leq s$ and $s \wedge t = 0$.
We may therefore assume that $s \wedge t \neq 0$.
But then $s \leq t$ or $t < s$.
The former cannot occur by assumption and so $t < s$.
It follows that $\dom (t) < \dom (s)$.
The semilattice of idempotents is $0$-disjunctive
and so there exists an idempotent $e < \dom (s)$ such that
$\dom (t)e = 0$.
Put $s' = se$.
Then $0 \neq s' \leq s$.
We have to calculate $s' \wedge t$.
Suppose that $a \leq s', t$.
Then $\dom (a) \leq \dom (s')\dom (t) = e\dom (t) = 0$, as required.
\end{proof}

A homomorphism $\theta \colon S \rightarrow T$ to a distributive inverse semigroup is said to be a {\em cover-to-join} map
if for each element $s \in S$ and each finite cover $A$ of $s$
we have that $\theta (s) = \bigvee \theta (A)$.
A detailed discussion of cover-to-join maps and how they originated from the work of Exel \cite{E} and Lenz \cite{Lenz}
can be found in \cite{Law8}.
If $\{e_{1}, \ldots, e_{m}\}$ is a cover of the idempotent $e$ and is also an orthogonal set then we say
that it is an {\em orthogonal cover}.
Our goal for the remainder of this section is to prove the following result.

\begin{proposition}\label{prop: bertha} In graph inverse semigroups which are pseudofinite, $0$-disjunctive and have no $0$-minimal elements,
the system of Lenz covers is determined by the sets $\hat{e}$ where $e$ is a maximal idempotent.
\end{proposition}

The following three lemmas will play an important role in simplifying our calculations.

\begin{lemma}\label{le: properties_of_covers} In an inverse semigroup, we have the following.
\begin{enumerate}

\item Let $a_{1}, \ldots, a_{m}$ be a set of elements below the element $a$.
Then $\{ \dom (a_{1}), \ldots, \dom (a_{m}) \}$ covers  $\dom (a)$
if and only if $\{ a_{1}, \ldots, a_{m} \}$ covers  $a$.

\item  Let $\{e_{1}, \ldots, e_{m} \}$ cover the idempotent $e$.
Suppose that $e \, \mathscr{D} f$ where $e = a^{-1}a$ and $f = aa^{-1}$. Then $\{ae_{1}a^{-1}, \ldots, ae_{m}a^{-1} \}$ covers $f$.

\item Let  $\{e_{1}, \ldots, e_{m} \}$ be a set of non-zero idempotents beneath the idempotent $e$.
Suppose that for any non-zero idempotent $f < e$ we have that $f \leq e_{i}$ for some $i$.
Then  $\{e_{1}, \ldots, e_{m} \}$ is a cover of $e$.

\end{enumerate}
\end{lemma} 
\begin{proof} (1) Suppose that  $\{ \dom (a_{1}), \ldots, \dom (a_{m}) \}$ covers $\dom (a)$.
Let $0 \neq b \leq a$.
Then $0 \neq \dom (b) \leq \dom (a)$.
By assumption, there exists $i$ such that $\dom (a_{i}) \wedge \dom (b) \neq 0$.
But $a_{i},b \leq a$ implies that $a_{i}$ and $b$ are compatible.
Thus by Lemma~\ref{le: basics},  we have that $\dom (a_{i} \wedge b) = \dom (a_{i}) \wedge \dom (b) \neq 0$.
Thus $a_{i} \wedge b \neq 0$, as required.
Conversely, suppose that  $\{ a_{1}, \ldots, a_{m} \}$ covers $a$.
Let $0 \neq e \leq  \dom (a)$.
Put $b = ae$.
Then $0 \neq b \leq a$.
By assumption, there exists $a_{i}$ such that $b \wedge a_{i} \neq 0$.
By the same reasoning as above,
we have that $0 \neq \dom (b) \wedge \dom (a_{i})$, that is,  $0 \neq e \wedge \dom (a_{i})$.

(2) Let $0 \neq p \leq f$.
Then $a^{-1}pa \leq e$ and it is easy to check that $a^{-1}pa \neq 0$.
By assumption, there exists an $i$ such that $a^{-1}pa \wedge e_{i} \neq 0$.
But  $a^{-1}pa \wedge e_{i} \leq a^{-1}a$ and so
$a(a^{-1}pa \wedge e_{i})a^{-1} \neq 0$.
But $a(a^{-1}pa \wedge e_{i})a^{-1} = p \wedge ae_{i}a^{-1} \neq 0$, as required, by Lemma~\ref{le: basics}.

(3) The definition of cover is trivially satisfied.
\end{proof}

\begin{lemma}\label{le: properties_of_maps} 
Let $\theta \colon S \rightarrow T$ be a homomorphism to a distributive inverse semigroup.
\begin{enumerate}

\item Suppose that for each non-zero idempotent $e$ and cover $F$ of $e$ we have that $\theta (e) = \bigvee_{f \in F} \theta (f)$.
Then $\theta$ is a cover-to-join map.

\item Let $G$ be a transversal of the non-zero $\mathscr{D}$-classes of the idempotents.
Suppose that for each idempotent $e \in G$ and cover $F$ of $e$ we have that $\theta (e) = \bigvee_{f \in F} \theta (f)$.
Then $\theta$ is a cover-to-join map.

\end{enumerate}
\end{lemma}
\begin{proof} (1) Let $\{a_{1}, \ldots, a_{m} \}$ be a cover of $a$.
Then $\{\dom (a_{1}), \ldots, \dom (a_{m})\}$ is a cover of $\dom (a)$ by part (1) of Lemma~\ref{le: properties_of_covers}. 
By assumption
$\theta (\dom (a)) = \bigvee_{i=1}^{m} \theta (\dom (a_{i}))$.
Now multiplying on the left by $\theta (a)$ and using distributivity we get that
$\theta (a) = \bigvee_{i=1}^{m} \theta (a_{i})$.

(2) Let $\{e_{1}, \ldots, e_{m} \}$ be a cover of the arbitrary non-zero idempotent $e$.
Suppose that $e = a^{-1}a$ and $f = aa^{-1}$ where $f \in G$.
Then $\{ae_{1}a^{-1}, \ldots, ae_{m}a^{-1} \}$ covers $f$  by part (2) of Lemma~\ref{le: properties_of_covers}. 
By assumption,
$\theta (f) = \bigvee_{i=1}^{m} \theta (ae_{i}a^{-1})$.
Multiply this equation by $\theta (a)^{-1}$ on the left and $\theta (a)$ on the right,
and we get
$\theta (e) = \bigvee_{i=1}^{m} \theta (e_{i})$.
The result now follows by part (1) above.
\end{proof}

\begin{lemma}\label{le: orthogonal_subcover} 
Let $S$ be an inverse semigroup with a semilattice of idempotents that is unambiguous.
Then a homomorphism $\theta$ from $S$ to a distributive inverse semigroup $T$ is a cover-to-join map if and only if 
for each idempotent $e \in G$ and orthogonal cover $F$ of $e$ we have that $\theta (e) = \bigvee_{f \in F} \theta (f)$.
\end{lemma}
\begin{proof} Let $\{e_{1}, \ldots, e_{m}\}$ be an arbitrary cover of the idempotent $e$.
We use the proof of Proposition~2.21 to obtain an orthogonal subset where each discarded element
is strictly less than one of the elements in the cover.
This too will be a cover.
The proof of the claim is now straightforward.
\end{proof}

The above three lemmas lead to the following corollary which is directly applicable to graph inverse semigroups.

\begin{corollary}\label{cor: max} Let $S$ be an inverse semigroup with maximal idempotents in which each non-zero idempotent
is $\mathscr{D}$-related to a maximal idempotent and whose semilattice of idempotents is unambiguous.
Then whether a homomorphism to a distributive inverse semigroup is a cover-to-join map or not
is determined solely by the orthogonal covers of the maximal idempotents
\end{corollary}

We now focus our attention on covers in graph inverse semigroups.
We assume that our graph $G$ has the property that the in-degree of each vertex is finite.
By Lemma~\ref{le: compactness}, we have that $P_{G}$ is pseudofinite.
Thus for each non-zero maximal idempotent $e$ we have that the set $\hat{e}$ is actually an orthogonal cover of $e$.
If $xx^{-1}$ is a non-zero idempotent in $P_{G}$ then we define its {\em weight} to be the number $\left| x \right|$.
In the lemma below, we shall denote generic elements of $P_{G}$ by bold letters to avoid confusion.

\begin{lemma}\label{le: frankie} Let $P_{G}$ be a graph inverse semigroup in which the in-degree of each vertex is finite and at least 2.
Let $F = \{ \mathbf{e}_{1}, \ldots, \mathbf{e}_{m} \}$ be an orthogonal cover of the maximal idempotent $\mathbf{e} = ee^{-1}$,
where $e$ is the identity associated with the vertex we shall denote by $e$.
Suppose that $\mathbf{e}_{1} = xx^{-1}$ is an idempotent in $F$ of maximum weight at least 1.
Put $x = \bar{x}a_{1}$ where $a_{1}$ is an edge with target $f$.
Let $a_{1}, \ldots, a_{n}$ be all the edges with target $f$.
\begin{enumerate}

\item Then $\mathbf{f}_{j} = \bar{x}a_{j}a_{j}^{-1}\bar{x}^{-1} \in F$ for $1 \leq j \leq n$.

\item Put $F' = F \setminus \{\mathbf{f}_{1}, \ldots, \mathbf{f}_{n} \} \cup \{\bar{x}\bar{x}^{-1} \}$.
Then $F'$ is an orthogonal cover of $e$ and $\left| F' \right| < \left| F \right|$.

\item $F = F' \setminus \{\bar{x}\bar{x}^{-1} \} \cup \bar{x} \hat{\mathbf{f}} \bar{x}^{-1}$ where $\mathbf{f} = ff^{-1}$.

\end{enumerate}
\end{lemma}
\begin{proof} (1) The string $\bar{x}a_{j}$ has target the vertex $e$.
Thus $\mathbf{f}_{j} = \bar{x}a_{j}a_{j}^{-1}\bar{x}^{-1} \leq \mathbf{e}$. 
By assumption $\mathbf{f}_{j} \wedge \mathbf{e}_{i} \neq 0$ for some $i$.
Let $\mathbf{e}_{i} = yy^{-1}$.
Then $y$ and $\bar{x}a_{j}$ are prefix-comparable.
By assumption, $\mathbf{e}_{1}$ has maximum weight amongst all the idempotents in $F$ and so
$\bar{x}a_{j} = yz$ for some path $z$.
If $z$ were not empty, $y$ would be a prefix of $\bar{x}$ and so we would have that $\mathbf{e}_{1} < \mathbf{e}_{i}$ which is a contradiction
on the assumption that our cover is orthogonal.
It follows that $z$ is empty and so $\mathbf{f}_{j } = \mathbf{e}_{i}$. 

(2) The following argument is enough to establish that $F'$ is a cover of $\mathbf{e}$.
Let $0 \neq \mathbf{i} \leq \mathbf{e}$.
Suppose that $0 \neq \mathbf{i} \wedge \mathbf{f}_{j}$ for some $j$.
Put $\mathbf{i} = zz^{-1}$.
Then $z$ and $\bar{x}a_{j}$ are prefix-comparable.
It is straightforward to check that this implies that  $z$ and $\bar{x}$ are prefix-comparable
and so $0 \neq \mathbf{i} \wedge \bar{x}\bar{x}^{-1}$.
Since $n \geq 2$ we have that  $\left| F' \right| < \left| F \right|$.

It remains to show that $F'$ is an orthogonal set.
Let $\mathbf{e}_{i}$ be an element not equal to any of the $\mathbf{f}_{J}$.
Let $\mathbf{e}_{i} = ww^{-1}$.
Suppose that $\bar{x}$ and $w$ are comparable.
There are two possibilities.
First, $w = \bar{x}z$ for some $z$ where $z$ is non-empty.
It follows that $w = \bar{x}a_{k}z'$ for some $k$ and string $z'$, possibly empty.
If $z'$ were non-empty then $w$ would be too long being longer than $x$ which is a contradiction.
We would therefore have to have $w =  \bar{x}a_{k}$ for some $k$ which would also be a contradiction
since $w$ would then equal one of the $\mathbf{f}_{j}$.
The only other possibility is that $\bar{x} = wz$ for some non-empty string $z$.
But this would imply that $\bar{x}$ was a proper prefix of $w$.
It follows that $\bar{x}a_{k}$ is a prefix of $w$ for some $k$.
But this would imply that $F$ was not an orthogonal set.
It follows that $F'$ is an orthogonal set.

(3) This is immediate.
\end{proof}

The following is a precise statement and proof of Proposition~\ref{prop: bertha}.

\begin{proposition}\label{prop: bert} Let $G$ be a directed graph where the in-degree of each vertex is finite and at least 2.
Let $\theta \colon P_{G} \rightarrow T$ be a homomorphism to a distributive inverse semigroup
where $\theta (e) = \bigvee_{f \in \hat{e}} \theta (f)$ for each maximal idempotent $e$ in $P_{G}$.
Then $\theta$ is a cover-to-join map.
\end{proposition}
\begin{proof} Let $e$ be an arbitrary maximal idempotent.
Then by Corollary~\ref{cor: max}, 
it is enough to relate arbitrary orthogonal covers of $e$ to the specific orthogonal cover $\hat{e}$.
This can be achieved using induction and Lemma~\ref{le: frankie} and,
since the maximal idempotents are pairwise orthogonal, we can fix attention on all covers of a fixed maximal idempotent $e$.
Suppose that our claim holds for all orthogonal covers of $e$ with at most $p$ elements.
Let $F$ be an orthogonal cover of $e$ with $p+1$ elements.
By Lemma~\ref{le: frankie}, we may write
$$F = F' \setminus \{\bar{x}\bar{x}^{-1} \} \cup \bar{x} \hat{f} \bar{x}^{-1}$$
where $F'$ is cover of $e$ and $\left| F' \right| < \left| F \right|$.
By our induction hypothesis, we may write
$$\theta (e) = \bigvee_{f' \in F'} \theta (f').$$
By assumption, we may write
$$\theta (f) = \bigvee_{g \in \hat{f}} \theta (g).$$
Thus using distributivity, we have that
$$\theta (\bar{x}\bar{x}^{-1}) = \bigvee_{g \in \hat{f}} \theta (\bar{x}g \bar{x}^{-1}).$$
It follows that 
$$\theta (e) = \bigvee_{f \in F} \theta (f),$$  
as required.
\end{proof}

\subsection{Distributive completions}

Let $S$ be an inverse semigroup with zero.
An {\em order ideal} of $S$ is any subset which is closed downwards under the natural partial order.
A {\em finitely generated order ideal} is subset of the form $\{s_{1}, \ldots, s_{m}\}^{\downarrow}$; that is, all those elements of $S$ that lie beneath at least one element $s_{i}$ 
for some $i$. 
The order ideal $\{s_{1}, \ldots, s_{m}\}^{\downarrow}$ is said to be {\em compatible} if the elements $s_{i}$ are pairwise compatible.
The set of all finitely generated compatible order ideals of $S$ is denoted by $\mathsf{D}(S)$. 
If $A,B \in \mathsf{D}(S)$ then define their product to be $AB$.
In this way, $\mathsf{D}(S)$ becomes an inverse semigroup.
The natural partial order is just subset inclusion and the join of compatible elements is their union.
Thus it is a distributive inverse semigroup.
If $S$ is a $\wedge$-semigroup then $\mathsf{D}(S)$ is a $\wedge$-semigroup:
the intersection of  $\{s_{1}, \ldots, s_{m}\}^{\downarrow}$ and $\{t_{1}, \ldots, t_{n}\}^{\downarrow}$
is  $\{s_{i} \wedge t_{j} \colon 1 \leq i \leq m, 1 \leq j \leq n  \}^{\downarrow}$.
There is a homomorphism $\delta \colon S \rightarrow \mathsf{D}(S)$ given by $s \mapsto s^{\downarrow}$.
The inverse semigroup $\mathsf{D}(S)$ is characterized by the following universal property;
it is the finitary version of one due to Boris Schein.
See Theorem~1.4.24 of \cite{Law1}.

\begin{theorem}\label{the: alf} Let $\theta \colon S \rightarrow T$ be a homomorphism to a distributive inverse semigroup $T$.
Then there is a unique homomorphism of distributive inverse semigroups $\theta^{\ast} \colon \mathsf{D}(S) \rightarrow T$ such that
$\theta^{\ast} \delta = \theta$.
\end{theorem}

In any inverse semigroup $S$ whose natural partial order is unambiguous, we find that
the semigroup  $\mathsf{D}(S)$ can appear in a different guise from the way it first did in \cite{Law22}.
We shall now explain how and in what way.

\begin{lemma}\label{le: art} 
Let $S$ be an inverse semigroup with an unambiguous partial order.
Every finitely generated compatible order ideal of $S$ has a unique orthogonal generating set excluding zero.
\end{lemma}
\begin{proof}
Let $\{s_{1}, \ldots, s_{m}\}^{\downarrow}$ be a finitely generated compatible order ideal of such a semigroup.
By Proposition~\ref{prop: important},  we may find a generating set for  $\{s_{1}, \ldots, s_{m}\}^{\downarrow}$ which is orthogonal.
Suppose now that 
$\{s_{1}, \ldots, s_{m}\}^{\downarrow} = \{t_{1}, \ldots, t_{n}\}^{\downarrow}$ 
where both sets of generators are orthogonal.
The element $s_{1}$ must be less than or equal to exactly one element on the righthandside.
Suppose that $s_{1} \leq t_{j}$.
By the same reasoning, $t_{j} \leq s_{i}$ for a unique element $s_{i}$ on the lefthandside.
But then we have $s_{1} \leq t_{j} \leq s_{i}$.
It follows that $s_{1} =s_{i} = t_{j}$.
Hence each element in the generating set on the lefthandside occurs exactly once amongst the generating set of the righthandside.
It follows by symmetry that the two generating sets are equal.
Thus in an inverse semigroup with an unambiguous natural partial order,
finitely generated compatible order ideals are determined by their orthogonal generating sets.
These sets are unique as long as zero is omitted.
\end{proof} 

The above result together with the fact that graph inverse semigroups have an unambiguous natural partial order
means that we may in fact work with $\mathsf{D}(S)$ rather than $\mathsf{O}(S)$ as was done in  \cite{Law21}.

\subsection{The construction of the Cuntz-Krieger semigroup}

Let $G$ be a directed graph.
We may construct the inverse semigroup $P_{G}$ as in Section~2.2 and therefore the inverse semigroup $D = \mathsf{D}(P_{G})$ as in Theorem~3.11.
By Lemma~3.12, each non-zero element $X$ of $D$ may be written uniquely as $X = A_{X}^{\downarrow}$ for a unique orthogonal subset $A_{X}$ of $X$.
We define $X \equiv Y$ if and only if  $A_{X} \leftrightarrow A_{Y}$.
We also put $\{0\} \equiv \{0\}$.

\begin{proposition} The relation $\equiv$ is a $0$-restricted, idempotent pure congruence on $\mathsf{D}(P_{G})$
and $\mathsf{D}(P_{G})/\equiv$ is a distributive inverse $\wedge$-semigroup.
\end{proposition}
\begin{proof} 
Let $A = \{a_{1}, \ldots, a_{m} \}^{\downarrow}$ and  $B = \{b_{1}, \ldots, b_{n} \}^{\downarrow}$
where  $\{a_{1}, \ldots, a_{m} \}$ and $ \{b_{1}, \ldots, b_{n} \}$ are othogonal sets.
We prove first that $A \equiv B$ if and only if $A \leftrightarrow B$ where this is calculated in $D$.
Suppose that  $A \leftrightarrow B$.
Let $0 \neq x \leq a_{i}$.
Then $0 \neq x^{\downarrow} \leq A$.
By assumption, $0 \neq x^{\downarrow} \wedge B$.
It follows that there is $b \in B$ such that $0 \neq x \wedge b$.
But $b \leq b_{j}$ for some $j$ and so $0 \neq x \wedge b_{j}$.
By symmetry, we have proved that  $A \equiv B$.
To prove the converse, suppose that  $A \equiv B$
and let $C =  \{c_{1}, \ldots, c_{p} \}^{\downarrow}$, where $ \{c_{1}, \ldots, c_{p} \}$ is an orthogonal set.
Let $0 \neq C \leq A$.
Choose $0 \neq c \in C$.
Then $c \leq a_{i}$ for some $i$.
It follows that $0 \neq c \wedge b_{j}$ for some $j$.
But $c \leq c_{k}$ for some $k$.
Hence  $0 \neq c_{k} \wedge b_{j}$.
We have therefore shown that $C \wedge B \neq 0$, as required.
By Lemma~\ref{le: lenz}, it follows that $\equiv$ is a $0$-restricted congruence.
The congruence is idempotent pure;
to see why, let $A \leftrightarrow B$ where $A$ contains only non-zero idempotents.
Let $xy^{-1} \in B$.
Then there exists $uu^{-1} \in A$ such that $xy^{-1} \wedge uu^{-1} \neq 0$.
But this implies that $xy^{-1}$ lies above a non-zero idempotent and $P_{G}$ is $E^{\ast}$-unitary.
It follows that $xy^{-1}$ is an idempotent.
Since $xy^{-1}$ was arbitrary, $B$  consists entirely of idempotents, as claimed.
\end{proof}

Define $CK_{G}$ to be $\mathsf{D}(P_{G})/\equiv$ and define $\xi \colon P_{G} \rightarrow CK_{G}$
by $\xi (s) = [s^{\downarrow}]$, where $[x]$ denotes the $\equiv$-class containing $x$.
We call  $CK_{G}$ the {\em Cuntz-Krieger semigroup of the directed graph $G$}.
The following theorem summarizes the key properties of this semigroup. 

\begin{theorem}\label{the: two} Let $G$ be a directed graph with finite in-degrees.
\begin{enumerate}

\item For any directed graph $G$, there is a distributive invese semigroup $CK_{G}$
together with a homomorphism $\xi \colon P_{G} \rightarrow CK_{G}$  such that every element of $CK_{G}$ is a join
of a finite non-empty compatible subset of the image of $\xi$.
For each maximal idempotent $e$ in $P_{G}$, we have that 
$$\delta (e) = \bigvee_{f \in \hat{e}} \delta (f).$$ 
The homomorphism $\xi$ is injective if and only if $G$ has the additional property that the in-degree of each vertex is either 0 or at least 2.

\item Let $\theta \colon P_{G} \rightarrow T$ be a homomorphism to a distributive inverse semigroup
where $\theta (e) = \bigvee_{f \in \hat{e}} \theta (f)$ for each maximal idempotent $e$ in $P_{G}$
Then there is a unique homomorphism of distributive inverse semigroups $\bar{\theta} \colon CK_{G} \rightarrow T$ such that
$\bar{\theta} \xi = \theta$.

\end{enumerate}
\end{theorem}
\begin{proof} (1). The semigroup $CK_{G}$ is a distributive inverse semigroup by the above proposition.
The homomorphism $\xi$ is injective if and only if the Lenz arrow relation is equality.
But $P_{G}$ is unambiguous and $E^{\ast}$-unitary, and it follows by Proposition~\ref{prop: separative} that the semilattice of idempotents of $P_{G}$ must be $0$-disjunctive.
By part (2) of Lemma~\ref{le: compactness}, 
this means that the in-degree of each vertex of $G$ is either 0 or at least 2.
The claim on maximal idempotents follows from part (3) of Lemma~\ref{le: properties_of_covers}.

(2). Put $S =  P_{G}$. 
By Proposition~\ref{prop: bert},   
the map $\theta$ is actually a cover-to-join map.
Any homomorphism to a distributive inverse semigroup factors through $\mathsf{D}(S)$ by Theorem~\label{the: alf}. 
There is therefore a unique homomorphism of distributive inverse semigroups $\theta^{\ast} \colon \mathsf{D}(S) \rightarrow T$ such that
$\theta^{\ast} \delta = \theta$.
the map $\theta^{\ast}$ is defined as follows:
$$\theta^{\ast}(\{s_{1}, \ldots, s_{m} \}^{\downarrow}) = \bigvee_{i=1}^{n} \theta (s_{i})$$
where we may, by Lemma~\ref{le: art}, assume that the set $\{s_{1}, \ldots, s_{m}\}$ is orthogonal.
Suppose that $X = \{s_{1}, \ldots, s_{m} \} \leftrightarrow \{t_{1}, \ldots, t_{n} \} = Y$
where $X$ and $Y$ are orthogonal subsets.
We prove that  $\bigvee_{x \in X} \theta (x) = \bigvee_{y \in Y} \theta (y)$.
By definition we have that $s_{i} \rightarrow \{t_{1}, \ldots, t_{n}\}$.
By Lemma~\ref{le: lenz_cover},
it follows that $\{s_{1} \wedge t_{j} \colon 1 \leq j \leq n\}$ is a cover of  $s_{i}$.
By assumption, we have that
$$\theta (s_{i}) = \bigvee_{j=1}^{n} \theta (s_{i} \wedge t_{j}).$$
It follows that $\theta (s_{i}) \leq \bigvee_{j=1}^{m} \theta (t_{j})$.
Hence
$$\bigvee_{i=1}^{m} \theta (s_{i}) \leq  \bigvee_{j=1}^{m} \theta (t_{j}).$$
The result now follows by symmetry.
Thus our map actually factors through $CK_{G} = \mathsf{D}(S)/\equiv$.
We therefore have defined a homomorphism $\bar{\theta} \colon CK_{G} \rightarrow T$
of distributive inverse semigroups such that $\bar{\theta} \xi = \theta$.
This map is unique because each element of  $CK_{G}$ is a join of elements in the image of $\xi$.
\end{proof}

In the following section, we shall prove that the semilattice of idempotents of a Cuntz-Krieger semigroup is a generalised Boolean algebra.
This will enable us to conclude that these semigroups are Boolean $\wedge$-semigroups.

\subsection{A representation by partial bijections}

Given a directed graph $G$, we define $G^{\omega}$ to be the set of all right-infinite paths in the graph $G$.
Such paths have the form $w = w_{1}w_{2}w_{3} \ldots$ where the $w_{i}$ are edges in the graph and $\dom (w_{i}) = \ran (w_{i+1})$.
If $x \in G^{\ast}$, and so is a finite path in $G$, we write $xG^{\omega}$ to mean the set of all
right-infinite paths in $G^{\omega}$ that begin with $x$ as a finite prefix.

\begin{lemma}\label{le: tom} 
If $xG^{\omega} \cap yG^{\omega} \neq \emptyset$ then $x$ and $y$ are prefix comparable
and so either $xG^{\omega} \subseteq yG^{\omega}$ or $yG^{\omega} \subseteq xG^{\omega}$.
\end{lemma}

If $G$ has any vertices of in-degree 0, that is, {\em sources}, then a finite path may get stuck and we may not be able to continue
it to an infinite path. For this reason, we shall require that our directed graphs have the property that the in-degree of each vertex is at least 1.
There is a map $G^{\ast} \rightarrow G^{\omega}$ given by $x \mapsto xG^{\omega}$.
It need not be injective but it will be useful to us to have a sufficient condition when it is.

\begin{lemma}\label{le: jerry} Let $G$ be a directed graph in which the in-degree of each vertex is at least 2.
Then if $x$ and $y$ are finite paths in the free category on $G$ such that $xG^{\omega} = yG^{\omega}$ then $x = y$.
\end{lemma}
\begin{proof} The finite paths $x$ and $y$ must be prefix comparable and have the same target vertex $v$.
Therefore to show that they are equal, it is enough to prove that they have the same length.
Without loss of generality assume that $\left| x \right| < \left| y \right|$.
Then $y = xz$ for some finite path $z$.
Denote the source vertex of $x$ by $u$.
Suppose that $z = a\bar{z}$ where $a$ is one edge with target $u$.
By assumption, there is at least one other edge with target $u$; call this edge $b$.
We may extend $b$ by means of an infinite path $\omega$.
Thus by assumption $xb\omega = yb\omega'$ for some infinite string $\omega'$.
But this implies that $b = a$ which is a contradiction.
It follows that $x$ and $y$ have the same length and so must be equal.
\end{proof}

We shall now define a topology on $G^{\omega}$.
The subsets of the form $xG^{\omega}$ where $x$ is a finite path in $G$ form a basis for a topology on $G^{\omega}$ by Lemma~\ref{le: tom}.
The topology that arises is Hausdorff and the sets $xG^{\omega}$ are compact.
More generally, the compact-open subsets of this topology are precisely the sets $AG^{\omega}$ where $A$ is a finite subset of $G^{\ast}$.
See \cite{KPRR} for details.
It follows that with this topology $G^{\omega}$ is a Boolean space --- a Hausdorff topological space with a basis of compact-open subsets ---
and the compact-open subsets of such a space form a generalized Boolean algebra under the usual set-theoretic operations. 

A {\em representation} of an inverse semigroup $S$ is a homomorphism $\theta \colon S \rightarrow I(X)$, 
where $I(X)$ is the symmetric inverse semigroup on the set $X$.
Equivalently, this can be defined by means of a partially defined action
where $s \cdot x$ is defined if and only if $x \in \mbox{dom} (\theta (s))$ in which case it is equal to $\theta (s)(x)$.
Abstractly, an action of this type is a partial function $S \times X \rightarrow X$ mapping $(s,x)$ to $s \cdot x$
when $\exists s \cdot x$ satisfying the following two axioms:
\begin{description}

\item[{\rm (A1)}] If $\exists e \cdot x$ where $e$ is an idempotent then $e \cdot x = x$.

\item[{\rm (A2)}] $\exists (st) \cdot x$ if and only if $\exists s \cdot (t \cdot x)$ in which case they are equal.

\end{description}
An action is {\em effective} if for each $x \in X$ there exists $s \in S$ such that $\exists s \cdot x$.

We now define an action of $P_{G}$ on the set $G^{\omega}$ of right-infinite paths in the graph $G$
and thereby define a homomorphism $\theta \colon P_{G} \rightarrow I(G^{\omega})$.
Let $xy^{-1} \in P_{G}$ and $w \in G^{\omega}$.
We define $xy^{-1} \cdot w$ if and only if we may factorize $w = yw'$ where $w' \in G^{\omega}$; in which case,  $xy^{-1} \cdot w = xw'$.
This is well-defined since $\dom (x) = \dom (y)$.
It is easy to check that the two axioms (A1) and (A2) for an action hold.
We call this action the {\em natural action of the graph inverse semigroup on the space of infinite paths}.

\begin{lemma} Let $\theta \colon P_{G} \rightarrow I(G^{\omega})$ be the homomorphism associated with the above natural action.
\begin{enumerate}

\item The action leads to a $0$-restricted homomorphism $\theta$ if and only if there is no vertex of in-degree 0.

\item If the in-degree of each vertex is at least 2 then the homomorphism $\theta$ is injective.

\end{enumerate}
\end{lemma}
\begin{proof} (1) Suppose that $\theta (xx^{-1}) = 0$ for some $x$.
This means that there are no right-infinite strings with prefix $x$.
This implies that there is some vertex of the graph which has in-degree zero.
On the other hand, if each vertex of the graph has in-degree at least one then the action is $0$-restricted:
given any finite path $x$ we may extend it to an infinite path $w = xw'$.
Then $xx^{-1} \cdot w$ is defined.

(2) Suppose that $\theta (xy^{-1}) = \theta (uv^{-1})$.
Then $yG^{\omega} = vG^{\omega}$ and so $y = v$ by Lemma~\ref{le: jerry}.
Similarly $xG^{\omega} = uG^{\omega}$ and so $x = u$ again by Lemma~\ref{le: jerry}.
\end{proof}

From now on, we shall assume that the in-degree of each vertex of the graph is (finite and) at least 2.
The representation $\theta \colon P_{G} \rightarrow I(G^{\omega})$ is injective and so
$P_{G}$ is isomorphic to its image $P'$.
Define $O_{G}$ to be the inverse subsemigroup of $I(G^{\omega})$ consisting of all non-empty finite unions of pairwise orthogonal elements of $P'$.

Let $X = \{x_{1}y_{1}^{-1}, \ldots, x_{m}y_{m}^{-1} \}$ be an orthogonal set in $P_{G}$.
Define a function $f_{X} \in I(G^{\omega})$ as follows
$$f_{X} \colon \bigcup_{i=1}^{m} y_{i}G^{\omega} \rightarrow \bigcup_{i=1}^{m} x_{i}G^{\omega}$$
where $f_{X}(w) = x_{i}w'$ if $w = y_{i}w'$ 

\begin{lemma} Let $X = \{x_{1}y_{1}^{-1}, \ldots, x_{m}y_{m}^{-1} \}$ and $Y = \{u_{1}v_{1}^{-1}, \ldots, u_{n}v_{n}^{-1} \}$ be two orthogonal sets in $P_{G}$.
Then $f_{X} = f_{Y}$ if and only if $X \leftrightarrow Y$.
\end{lemma} 
\begin{proof} By definition
$$f_{X} \colon \bigcup_{i=1}^{m} y_{i}G^{\omega} \rightarrow \bigcup_{i=1}^{m} x_{i}G^{\omega}$$
and
$$f_{Y} \colon \bigcup_{j=1}^{n} v_{j}G^{\omega} \rightarrow \bigcup_{j=1}^{n} u_{j}G^{\omega}$$
We suppose first that $f_{X} = f_{Y}$.
Thus
$$\{y_{1}, \ldots, y_{m} \}G^{\omega} = \{v_{1}, \ldots, v_{n} \}G^{\omega}
\text{ and }
\{x_{1}, \ldots, x_{m} \}G^{\omega} = \{u_{1}, \ldots, u_{n} \}G^{\omega}.$$

Let $0 \neq wz^{-1} \leq x_{i}y_{i}^{-1}$.
Then for some finite string $p$ we have that $w = x_{i}p$ and $z = y_{i}p$.
By definition, $f_{X}$ restricts to define a map from $x_{i}pG^{\omega}$ to $y_{i}pG^{\omega}$ 
such that for any infinite string $\omega$ for which the product is defined we have that $f_{X}(y_{i}p\omega  ) = x_{i}p \omega$.
By assumption, $f_{Y}(y_{i}p\omega) =  x_{i}p \omega$.
It follows that there are two possibilities.
Either $zG^{\omega}$ has a non-empty intersection with $v_{j}G^{\omega}$ with exactly one of the $j$,
in which case $zG^{\omega} \subseteq v_{j}G^{\omega}$ or it intersects a number of them
in which case $v_{j}G^{\omega} \subseteq zG^{\omega}$ for a number of the $j$.

Suppose the first possibility occurs.
Then $z = v_{j}q$ for some finite path $q$.
The map from $zG^{\omega}$ to $wG^{\omega}$ must be a restriction of the map from $v_{j}G^{\omega}$ to $u_{j}G^{\omega}$.
It follows that $wG^{\omega} = u_{j}qG^{\omega}$ and so by the above lemma we have that $w = u_{j}q$.
It follows that $wz^{-1} \leq u_{j}v_{j}^{-1}$.

We now suppose that the second possibility occurs.
Then for at least one $j$ we have that $v_{j} = zq$ for some finite path $q$.
In this case, we have that $u_{j}G^{\omega} = wqG^{\omega}$ and so by Lemma~\ref{le: jerry}
we have that $u_{j} = wq$.
It follows that $u_{j}v_{j}^{-1} \leq wz^{-1}$.

We have therefore shown that $X \rightarrow Y$.
The result follows by symmetry.

We now prove the converse.
Suppose that $X \leftrightarrow Y$.
We prove that $f_{X} = f_{Y}$.
Let $w$ be an infinite string in $\mbox{dom}(f_{X})$.
Then we may write it as $w = y_{i}\bar{w}$ for some infinite string $\bar{w}$.
By definition $f_{X}(y_{i} \bar{w}) = x_{i}\bar{w}$.
For every prefix $p$ of $\bar{w}$, the element $x_{i}p(y_{i}p)^{-1} \leq x_{i}y_{i}^{-1}$ 
and so there is an element $u_{j}v_{j}^{-1} \in Y$, depending on $p$,
such that $x_{i}p(y_{i}p)^{-1} \wedge  u_{j}v_{j}^{-1} \neq 0$.
This meet is equal to $x_{i}p(y_{i}p)^{-1}$ or $u_{j}v_{j}^{-1}$, whichever is smaller.
Therefore if we choose $p$ sufficiently long, 
we can ensure that the element $x_{i}p(y_{i}p)^{-1}$ cannot be greater than or equal to any element in $Y$.
It follows that there is a $j$ such that $x_{i}p(y_{i}p)^{-1} \leq u_{j}v_{j}^{-1}$ using the fact that $X \rightarrow Y$.
Thus $x_{i}p = u_{j}q$ and $y_{i}p = v_{j}q$ for some finite path $q$.
Put $\bar{w} = pw'$.
Then $w = y_{i}pw'$.
Thus 
$$f_{Y}(w) = f_{Y}(y_{i}pw') = f_{Y}(v_{j}qw' ) = u_{j}qw' = x_{i}pw' = x_{i}\bar{w} = f_{X}(w).$$
The result now follows by symmetry.
\end{proof}

It follows by the above lemma that the function $F \colon CK_{G} \rightarrow O_{G}$ given by $F([A^{0}]) = f_{A}$ is well-defined and a bijection.

\begin{theorem} Let $G$ be a directed graph in which the in-degree of each vertex is at least 2 and is finite.
Then the inverse semigroup $CK_{G}$ is isomorphic to the inverse semigroup $O_{G}$ defined as an inverse semigroup of partial bijections of the set $G^{\omega}$.
In particular, the semilattice of idempotents of $CK_{G}$ is a generalized Boolean algebra.
It follows that $CK_{G}$ is a Boolean inverse $\wedge$-semigroup.
\end{theorem}
\begin{proof}
We have defined a bijection from $F \colon CK_{G} \rightarrow O_{G}$.
It remains to show that this is a homomorphism.
From \cite{Law1}, we can simplify this proof by splitting it up into three simple cases.\\

\noindent
(Case~1) Check that if $\dom (X) = \ran (Y)$ then $F(XY) = F(X)F(Y)$.\\
Let $X = \{x_{1}y_{1}^{-1}, \ldots, x_{m}y_{m}^{-1} \}$ and $Y = \{u_{1}v_{1}^{-1}, \ldots, u_{n}v_{n}^{-1} \}$.
By assumption,  $\dom (X) = \ran (Y)$ and so $m = n$ and $y_{i}y_{i}^{-1} = u_{i}u_{i}^{-1}$.
It follows that 
$$XY = \{x_{1}v_{1}^{-1}, \ldots, x_{n}v_{n}^{-1} \},$$
remembering orthogonality.
It is now easy to check that $F(XY) = F(X)F(Y)$.\\

\noindent  
(Case~2) Check that if $X \leq Y$ then $F(X) \leq F(Y)$.\\
Suppose that $X \leq Y$.
Let $xy^{-1} \in X$.
Then $xy^{-1} \leq uv^{-1}$ for some $uv^{-1} \in Y$.
Thus $x = up$ and $y = vp$ for some finite path $p$.
It follows that $xG^{\omega} \subseteq uG^{\omega}$ and $yG^{\omega} \subseteq vG^{\omega}$.
It is simple to check that $F(xy^{-1}) \leq F(uv^{-1})$.
Thus by glueing the separate functions together we have that $F(X) \leq F(Y)$.\\ 

\noindent
(Case~3) Check that if $X$ and $Y$ are idempotents then $F(X \wedge Y) = F(X)F(Y)$.\\
Observe first that $xx^{-1} \leq yy^{-1}$ if and only if $xG^{\omega} \subseteq yG^{\omega}$
and that $xx^{-1} \wedge yy^{-1} = zz^{-1}$ if and only if $xG^{\omega} \cap yG^{\omega} = zG^{\omega}$.
Let $X = \{x_{1}x_{1}^{-1}, \ldots, x_{m}x_{m}^{-1} \}$ and $Y = \{u_{1}u_{1}^{-1}, \ldots, u_{n}u_{n}^{-1} \}$.
Then $X \wedge Y$ is constructed by forming all possible meets $x_{i}x_{i}^{-1} \wedge u_{j}u_{j}^{-1}$.
But this translates into forming all possible intersections $x_{i}G^{\omega} \cap u_{j}G^{\omega}$.
The result is now clear.\\  

The semilattice of idempotents of $O_{G}$ is in bijective correspondence with the subsets of $G^{\omega}$
of the form $XG^{\omega}$ where $X$ is a finite set of finite paths in $G$.
These are precisely the compact-open subsets of the topological space $G^{\omega}$
and form a generalized Boolean algebra.
\end{proof}

\begin{remark}
{\em Any $0$-restricted homomorphism $\theta \colon P_{G} \rightarrow I(X)$ is called a {\em strong representation} of $P_{G}$
if for each maximal idempotent $e$ of $P_{G}$ we have that $\theta (e) = \bigvee_{f' \in \hat{e}} \theta (f')$.
Strong representations of just polycyclic inverse monoids are already interesting; see, \cite{B, JL, K1, K2, Law1, Law3}.
By the theory we have developed,
every strong representation  $\theta \colon P_{G} \rightarrow I(X)$ gives rise to a
homomorphism $\bar{\theta} \colon CK_{G} \rightarrow I(X)$ of distributive inverse semigroups.
The relationship between $CK_{G}$ and $P_{G}$ is therefore analogous to the relationship between
the group-ring of a group and the group itself;
the partial join operation on distributive inverse semigroups gives them a ring-like character.}
\end{remark}

\subsection{The Cuntz-Krieger semigroup as an ample semigroup}

This topic is taken up in more depth in \cite{Law8}.
Here we shall just sketch out the key result.
We shall use the following notation.
Let $w = uw'$ where $w,w'$ are infinite strings and $u$ is a finite string.
Define $u^{-1}w = w'$.
Given a directed graph $G$ a groupoid $\mathcal{G}$ is defined as follows.
Its elements consist of triples $(w,k,w') \in G^{\omega} \times \mathbb{Z} \times G^{\omega}$ where
there are finite strings $u$ and $v$ such that $u^{-1}w = v^{-1}w'$ and $\left| v \right| - \left| u \right| = k$.
The groupoid product is given by $(w,k,w')(w',l,w'') = (w,k+l,w'')$ and $(w,k,w')^{-1} = (w',-k,w)$.
A basis for a topology is given as follows.
For each pair $x,y \in G^{\ast}$ define $Z(x,y)$ to consist
of all groupoid elements $(w,k,w')$ where $x^{-1}w = y^{-1}w'$ and $k = \left| y \right| = \left| x \right|$.
Observe that under our assumptions on $G$, the sets $Z(x,y)$ are always non-empty.
It can be shown that this is a basis, and that with respect to the topology that results the groupoid
$\mathcal{G}$ is an \'etale, Hausdorff topological groupoid in which the sets $Z(x,y)$ are compact-open bisections.
The space of identities of this groupoid is homeomorphic to the usual topology defined on $G^{\omega}$ \cite{KPRR}.
With our usual assumptions on the directed graph $G$ this makes $\mathcal{G}$ what we have called a {\em Hausdorff Boolean groupoid} in \cite{LL}.
The compact-open bisections of the groupoid $\mathcal{G}$ form an inverse semigroup called the {\em ample semigroup} of $\mathcal{G}$.
We shall prove that this semigroup is the Cuntz-Krieger semigroup $CK_{G}$.
The proof of part (1) of the following lemma is straightforward and part (2) is Lemma~2.5 of \cite{KPRR}.

\begin{lemma}\label{le: bertha} With the above notation, the following hold:
\begin{enumerate}

\item
$$
Z(x,y)Z(u,v) = \left\{
\begin{array}{ll}
Z(x,vz) & \mbox{if $y = uz$} \\
Z(yz,v) & \mbox{if $u = yz$}\\
\emptyset & \mbox{else}
\end{array}
\right.
$$

\item
$$
Z(x,y) \cap Z(u,v) = \left\{
\begin{array}{ll}
Z(x,y) & \mbox{if $xy^{-1} \leq  uv^{-1}$} \\
Z(u,v) & \mbox{if $uv^{-1} \leq xy^{-1}$}\\
\emptyset & \mbox{else}
\end{array}
\right.
$$
\end{enumerate}
\end{lemma}

Denote by $\mathsf{B}(\mathcal{G})$ the inverse semigroup of compact-open bisections of the topological groupoid $\mathcal{G}$.

\begin{theorem} The Cuntz-Krieger semigroup $CK_{G}$ is the ample semigroup of the topological
groupoid $\mathcal{G}$ constructed from the directed graph $G$.
\end{theorem}
\begin{proof}
Define $\theta \colon P_{G} \rightarrow \mathsf{B}(\mathcal{G})$ by $\theta (xy^{-1}) = Z(x,y)$
and map the zero to the emptyset.
Then by Lemma~\ref{le: bertha}, this map is a homomorphism.
We claim that it is injective.
Suppose that $Z(x,y) = Z(u,v)$
Since these sets are non-empty, we have that $xy^{-1}$ and $uv^{-1}$ are comparable since the poset $P_{G}$ is unambiguous.
It is now immediate by Lemma~\ref{le: bertha} that $xy^{-1} = uv^{-1}$.
Let $e$ be a vertex of $G$ and let $a_{1}, \ldots, a_{m}$ be the edges of $G$ with source $e$.
Then
$$Z(1_{e},1_{e}) = \bigcup_{i=1}^{m} Z(a_{i},a_{i}).$$
The conditions of Theorem~\label{the: two}  hold and so $\theta$ may be extended to a homomorphism 
$\bar{\theta} \colon CK_{G} \rightarrow \mathsf{B}(\mathcal{G})$.
Lemma~\ref{le: bertha} implies that each element of $\mathsf{B}(\mathcal{G})$ is a finite disjoint union of basis elements and so $\theta$ is surjective.

It remains to show that $\bar{\theta}$ is injective.
We shall prove that if
$$\bigcup_{i=1}^{m} Z(x_{i},y_{i}) = \bigcup_{j=1}^{n} Z(u_{j},v_{j})$$
then
$$\{x_{1}y_{1}^{-1}, \ldots, x_{m}y_{m}^{-1} \} \leftrightarrow \{u_{1}v_{1}^{-1}, \ldots, u_{n}v_{n}^{-1}\}.$$
By symmetry it is enough to prove that
if 
$$Z(x,y) \subseteq \bigcup_{j=1}^{n} Z(u_{j},v_{j})$$
then
$$xy^{-1} \rightarrow (u_{1}v_{1}^{-1}, \ldots, u_{n}v_{n}^{-1}).$$
Let $wz^{-1} \leq xy^{-1}$.
Then $w = xp$ and $z = yp$ for some finite path $p$.
Let $w'$ be any infinite path so that $xpw'$ and $ypw'$ are defined.
Then 
$$(xpw', \left| y \right| - \left| x \right|, ypw') \in Z(x,y)$$
and so belongs to $Z(u_{j},v_{j})$ for some $j$.
It is now easy to show that $wz^{-1}$ and $u_{j}v_{j}^{-1}$ are comparable
where we make essential use of the central number in the triple.
\end{proof}

\begin{remark} 
{\em We have had to place extra conditions on the graph in order that we can get a faithful
representation of the Cuntz-Krieger semigroup on the set of right-infinite strings.
The general case is discussed in \cite{LL}.}
\end{remark}


\end{document}